\documentclass[leqno]{amsart}
\usepackage{cases}
\usepackage{cases}
\usepackage{amsmath}
\usepackage{amsmath}
\usepackage{amsmath}
 \usepackage{graphicx}
 \usepackage{mathrsfs}
\usepackage{bbm}
\usepackage{amsfonts}
\usepackage{amssymb}
\usepackage{ams,amsmath,amsthm}

\usepackage{hyperref}

\def\opn#1#2{\def#1{\operatorname{#2}}} 
\opn\chara{char} \opn\length{\ell}
\opn\projdim{proj\,dim} \opn\injdim{inj\,dim} \opn\rank{rank}
\opn\depth{depth} \opn\grade{grade} \opn\height{height}
\opn\embdim{emb\,dim} \opn\codim{codim}

\opn\Tr{Tr} \opn\bigrank{big\,rank}
\opn\superheight{superheight}\opn\lcm{lcm}
\opn\trdeg{tr\,deg}%
\opn\reg{reg} \opn\lreg{lreg}
\opn\Ker{Ker} \opn\Coker{Coker} \opn\Im{Im} \opn\Hom{Hom}
\opn\Tor{Tor} \opn\Ext{Ext} \opn\End{End} \opn\Aut{Aut} \opn\id{id}

\opn\nat{nat}
\opn\pff{pf}
\opn\Pf{Pf} \opn\GL{GL} \opn\SL{SL} \opn\mod{mod} \opn\ord{ord}

\def\Implies{\ifmmode\Longrightarrow \else
     \unskip${}\Longrightarrow{}$\ignorespaces\fi}
\def\implies{\ifmmode\Rightarrow \else
     \unskip${}\Rightarrow{}$\ignorespaces\fi}
\def\iff{\ifmmode\Longleftrightarrow \else
     \unskip${}\Longleftrightarrow{}$\ignorespaces\fi}

\let\:=\colon

\newtheorem{Theorem}{Theorem}[section]
\newtheorem{Lemma}[Theorem]{Lemma}

\newtheorem{Remark}[Theorem]{Remark}

\newtheorem{Example}[Theorem]{Example}

\newtheorem{Algorithm}[Theorem]{Algorithm}
\newtheorem{Assumption}[Theorem]{Assumption}

\theoremstyle{definition}

\opn\ini{in} \opn\inm{inm} \opn\Sym{Sym} \opn\diag{diag}
\opn\Ii{(i)} \opn\Iii{(ii)}

\title[Convergence of AFEM for OCPs in $L^2$-norm]{ Convergence of $L^2$-norm based adaptive finite element method for elliptic optimal control problems}
\author{Wei Gong $^*$}
\author{Ningning Yan $^\diamond$}
\author{Zhaojie Zhou $^\dag$}
\thanks{$^*$ NCMIS, LSEC, Institute of Computational Mathematics, Academy of Mathematics and Systems Science, Chinese Academy of Sciences, Beijing 100190, China.
 Email: {\tt wgong@lsec.cc.ac.cn}}
\thanks{$^\diamond$ LSEC, Institute of Systems Science, Academy of Mathematics and Systems Science, Chinese Academy of Sciences, Beijing 100190, China.
 Email: {\tt ynn@amss.ac.cn}}
\thanks{$^\dag$ School of Mathematical Sciences, Shandong Normal University, Jinan 250014, China.
 Email: {\tt zzj534@amss.ac.cn}}
\date{Today is \today}

\begin{document}
\maketitle

{\bf Abstract:}\hspace*{10pt} { This paper aims to study the convergence of adaptive finite element method for control constrained elliptic optimal control problems under $L^2$-norm. We prove the contraction property and quasi-optimal complexity for the $L^2$-norm errors of both the control, the state and adjoint state variables with $L^2$-norm based AFEM, this is in contrast to and improve our previous work \cite{Gong_Yan} where convergence of AFEM based on energy norm had been studied and suboptimal convergence for the control variable was obtained and numerically observed. For the discretization we use variational discretization for the control and piecewise linear and continuous finite elements for the state and adjoint state. Under mild assumptions on the initial mesh and the mesh refinement algorithm we prove the optimal convergence of AFEM for the control problems, numerical results are provided to support our theoretical findings.    }

{{\bf Keywords:}\hspace*{10pt} optimal control problem,
 elliptic equation, control constraints, adaptive finite element method, $L^2$-norm error estimates, convergence and optimality }


\section{Introduction}
\setcounter{equation}{0}


In this paper we intend to prove the convergence of adaptive finite element method (AFEM for shot) for solving optimal control problems (OCPs) governed by partial differential equations. The adaptive finite element procedure for solving OCPs  consists of the following loop
$$\mbox{SOLVE}\rightarrow \mbox{ESTIMATE}\rightarrow \mbox{MARK}\rightarrow \mbox{REFINE}.$$
The SOLVE step outputs the finite element solutions of the optimal control problems by solving the reduced optimization problems with appropriate optimization algorithms. The ESTIMATE step is based on the a posteriori error estimators which should be reliable and may also be efficient, while the step MARK selects the set of elements for refinement  by some strategies, like MAX strategy (\cite{Mekchay}) or D\"orfler's marking strategy (\cite{Dorfler}), based on the error indicators derived from the contributions of the control, the state and adjoint state approximations. The step REFINE can be done by using iterative or recursive bisection of elements with the minimal refinement condition (see \cite{Stevenson,Verfuth}). 

Nowadays adaptive finite element method is viewed as one of the most efficient methods for solving partial differential equations and has been proved to possess optimal computational complexity.  We refer to \cite{Dorfler,Binev,Morin,Nochetto,Mekchay,Cascon} for convergence analysis and \cite{Binev,Stevenson,Rob,Cascon} for optimal cardinality, which provide solid theoretical support for the success of AFEM when applied to solve second order elliptic equations. The applications of AFEM to optimal control problems differ from the error estimators used for  the adaptive mesh refinement, here we mention the work \cite{Liu2} of Liu, Yan for residual type a posteriori error estimates and \cite{Becker} of Becker, Kapp, Rannacher  for dual-weighted goal-oriented adaptivity. We also refer to \cite{Hoppe,Hoppe1,Li,Liu5,Liu3,Liu4,LiuYan08book} for the extensions and applications in distributed and boundary control problems, Stokes control problems, time-dependent control problems and so on. To prove the convergence and optimality of AFEM we require both the reliability and efficiency of the error estimators. Recently, Kohls, R\"{o}sch and Siebert derived in \cite{Kohls} an error equivalence property with respect to the $L^2$-norm error for the control and energy norm errors for the state and adjoint state, which helps to derive reliable and efficient a posteriori error estimators for optimal control problems with either variational discretization or full control discretization.


Although the convergence theory of AFEM for boundary value problems is almost complete, the convergence results of AFEM for solving optimal control problems are scarce and far from satisfactory. Here we give some comments on existing results. Gaevskaya et. al studied in \cite{Gaevskaya} the convergence of AFEM for OCPs with piecewise constant control discretization. They proved an error reduction property for the optimal control, state,  adjoint state and the associated co-control variables with  some additional requirements on the strict complementarity of the continuous problem and the non-degeneracy property of the discrete control problem, and the marking strategy was extended to include the discrete free boundary between the active and inactive control sets. Becker and Mao (\cite{Mao}) gave a convergence proof for the adaptive algorithm by viewing the control problems as a nonlinear elliptic system of the state and adjoint variables, the adaptive algorithm presented there involved  the marking of data oscillation. In  \cite{Kohls1} the authors proved the plain convergence of the adaptive algorithm without   convergence rate and optimality. Recently, Gong and Yan (\cite{Gong_Yan}) gave a rigorous convergence proof for the adaptive finite element algorithm of elliptic optimal control problem in an optimal control framework. The main idea is to show that the energy norm errors of the state and adjoint state variables are equivalent to the boundary value approximations of the state and adjoint state equations up to a higher order term, so that the standard convergence results of AFEM for elliptic boundary value problems can be used. Later on, the ideas were used to prove the convergence of an adaptive multilevel correction finite element method for elliptic optimal control problem. We also mention that in \cite{Leng} Chen and Leng proved the convergence and quasi-optimality of AFEM for an elliptic control problem with integral type control constraint by using piecewise constant control discretization.

We remark that all the results mentioned above are based on AFEM in energy norm error for both boundary value problems and OCPs. The motivation to study the convergence of $L^2$-norm based AFEM for solving OCPs in current paper is two fold. Firstly, it is of great theoretical importance to prove the convergence and optimality of AFEM for the control variable.  In \cite{Gong_Yan} the authors showed that the convergence of AFEM based on energy norm was suboptimal for the control and the numerical experiments confirmed this sub-optimality. Recall that in the a priori error estimates for optimal control problems (\cite{Hinze05COAP}), one can obtain optimal convergence of the control variable by using only optimal $L^2$-norm error estimates for the state and the adjoint state variables. This strongly suggests a posteriori error estimates and adaptive algorithm in $L^2$-norm. Secondly, it is of practical use to study $L^2$-norm based adaptive finite element method for OCPs. In \cite{Schneider} the authors considered a posteriori error estimates for elliptic optimal control problems in $L^2$-norm by extending the ideas of \cite{Kohls}. It has been pointed out in \cite{Schneider} that $L^2$-norm based error estimators deliver better reduction of the error $\|u-u_h\|_{0,\Omega}$ and  mesh node distribution than energy norm based error estimator, where $u$ and $u_h$ are continuous and discrete control variables, respectively. 

Since the Galerkin approximation of second order elliptic equation is only the best approximation in energy norm, it is not straightforward to prove the convergence of AFEM in $L^2$-norm. Here we mention some attempts to prove convergence of AFEM under weaker norms other than the global energy norm. Morin et al. \cite{Veeser} obtained plain convergence of AFEM for controlling weak norms under quite general assumptions on the marking strategy and norm of interest. Demlow studied in  \cite{Demlow_2010} the convergence of AFEM under local energy norm error. Demlow and Steveson proved in \cite{Demlow} the contraction property and optimal convergence rate of AFEM for controlling $L^2$-norm with D\"orfler's marking strategy under appropriate mesh grading conditions. The results of \cite{Demlow} were then used to prove the quasi-optimality of adaptive finite element methods for controlling local energy errors in \cite{Demlow_2015}. The convergence analysis of AFEM in $L^2$-norm presented in \cite{Demlow} is based on the equivalence between the weighted energy norm error and $L^2$-norm error under additional assumption on the mesh grading. One also need to impose $H^2$-regularity for the underline elliptic equation for the derivation of efficient and reliable a posteriori error estimates. 

In this paper we borrow some ideas of \cite{Demlow} to prove the convergence of $L^2$-norm based AFEM for OCPs. Here we consider a general second order elliptic equation with variable coefficients other than the Laplacian in \cite{Demlow}. We remark that the application of results in \cite{Demlow} to OCPs is not trivial as we need to deal with the coupling of the state, the adjoint and the control variables in an appropriate way. Moreover, compared to \cite{Gong_Yan} we do not use the equivalence between the optimal control approximation and the boundary value approximations but follow the standard approaches to prove the convergence of AFEM as done in \cite{Cascon}. We show that the finite element solutions of the optimal control problem are nearly best approximations to the continuous ones from the finite element space in $L^2$-norm up to an oscillation term. Moreover, we show the contraction property and quasi-optimal complexity for the $L^2$-norm errors of both the control, the state and adjoint state variables with $L^2$-norm based AFEM, this improves the known result of \cite{Gong_Yan} for energy norm based AFEM. In our opinion, the results obtained in current paper together with \cite {Gong_Yan}  complete the convergence theory of AFEM for solving elliptic OCPs with variational control discretization.


The rest of the paper is organized as follows. In Section 2 we introduce the model elliptic optimal control problem and its finite element approximation, we also derive efficient and reliable a posteriori error estimates in $L^2$-norm. The adaptive algorithm based on D\"{o}rfler's marking strategy is also presented. In Section 3 we give some connections between the weighted energy norm errors and the $L^2$-norm errors, which enable us to give a convergence analysis of the AFEM for optimal control problems, the quasi-optimal cardinality is proved in Section 4. Numerical experiments are carried out in Section 5 to validate our theoretical result. 

Let $\Omega\subset\mathbb{R}^d$ ($d=2,3$) be a bounded, convex polygonal or polyhedral domain. We denote by $W^{m,q}(\Omega)$ the usual Sobolev space of order $m\geq 0$, $1\leq q<\infty$ with norm $\|\cdot\|_{m,q,\Omega}$. For $q=2$ we denote $W^{m,q}(\Omega)$ by $H^m(\Omega)$ and $\|\cdot\|_{m,\Omega}=\|\cdot\|_{m,2,\Omega}$, which is a Hilbert space. We set $(\cdot,\cdot)$ as the inner product in $L^2(\Omega)$. We denote $C$ a generic positive constant which may stand for different values at its different occurrences but does not depend on mesh size. We use the symbol $A\lesssim B$ to denote $A\leq CB$ for some constant $C$ that is independent of mesh size. If in addition $B\lesssim A$ holds we use the symbol $A\simeq B$.

\section{Adaptive finite element method for the optimal control problem}
\setcounter{equation}{0}
In this section we consider the following elliptic optimal control problem:
\begin{eqnarray}
\min\limits_{u\in U_{ad}}\ \ J(y,u)={1\over
2}\|y-y_d\|_{0,\Omega}^2 +
\frac{\alpha}{2}\|u\|_{0,\Omega}^2\label{OCP}
\end{eqnarray}
subject to
\begin{equation}\label{OCP_state}
\left\{\begin{array}{llr}
\mathcal{L} y=f+u \quad&\mbox{in}\
\Omega, \\
\ \ y=0  \quad &\mbox{on}\ \partial\Omega,
\end{array}\right.
\end{equation}
where $\alpha>0$ is a fixed parameter, $f$ is a given function, $y_d\in L^2(\Omega)$ is the desired state and the linear second order elliptic operator $\mathcal{L}$ is defined as follows:
\begin{eqnarray}
\mathcal{L}y:=-\sum\limits_{i,j=1}^d\frac{\partial}{\partial x_j}(a_{ij}\frac{\partial y}{\partial x_i})+a_0y.\nonumber
\end{eqnarray}
We require that $0\leq a_0<\infty$, $a_{ij}\in W^{1,\infty}(\Omega)$ ($i,j=1,...,d$) and $(a_{ij})_{d\times d}$ is symmetric and positive definite. We set $A=(a_{ij})_{d\times d}$ and $A^*$ its adjoint. We also denote $\mathcal{L}^*$ the adjoint operator of $\mathcal{L}$. Moreover, $U_{ad}$ is the admissible control set with bilateral control constraints:
\begin{eqnarray}
U_{ad}:= \Big\{u\in L^2(\Omega), \quad a\leq u\leq b\ \mbox{a.e.}\ \mbox{in}\  \Omega\Big\},\nonumber
\end{eqnarray}
where $a,b\in \mathbb{R}$ and $a<b$. 

We introduce the affine linear control-to-state mapping $S:L^2(\Omega)\rightarrow H_0^1(\Omega)$ such that for each $f+u\in L^2(\Omega)$ we associate the unique solution of problem (\ref{OCP_state}) with $y=S(f+u)$, i.e.,
\begin{eqnarray}
a(y,v)=(f+u,v)\quad \forall v\in H_0^1(\Omega),\label{OCP_state_weak}
\end{eqnarray}
where
\begin{eqnarray}
a(y,v):=(A\nabla y,\nabla v)+(a_0y,v),\quad \forall y,v\in H_0^1(\Omega).\nonumber
\end{eqnarray}
We denote $\||v\||:=\sqrt{a(v,v)}$ for the global energy norm over $\Omega$ and $\||v\||_D$  the local energy semi-norm when the integral is restricted to  $D\subset\Omega$.

Since $f$ is fixed, we use $y=Su$ instead of $y=S(f+u)$ in the following to emphasize the dependence on $u$. Then we are led to a reduced optimization problem $\min\limits_{u\in U_{ad}}\ \hat J(u):=J(Su,u)$ involving only the control $u$. By standard arguments (\cite{Lions}) we can prove that the above optimization problem admits a unique solution. Moreover, we can derive the first order optimality condition: 
\begin{eqnarray}
(\alpha u+p,v-u)\geq 0\quad \forall v\in U_{ad},\label{OCP_opt}
\end{eqnarray}
where the adjoint state $p\in H_0^1(\Omega)$ satisfies
\begin{equation}\label{OCP_adjoint}
\left\{\begin{array}{llr}
\mathcal{L}^* p=y-y_d \quad&\mbox{in}\
\Omega, \\
 \quad \ p=0  \quad &\mbox{on}\ \partial\Omega.
\end{array}\right.
\end{equation}
It is clear that $p=S^*(y-y_d)$ where $S^*$ is the adjoint operator of $S$ such that
\begin{eqnarray}
a(v,p)=(y-y_d,v)\quad \forall v\in H_0^1(\Omega).\label{OCP_adjoint_weak}
\end{eqnarray}
From (\ref{OCP_opt}) we can derive the pointwise representation of the control $u$ through the adjoint $p$: $u=P_{[a,b]}(p):=\max\{a,\min\{b,-\frac{p}{\alpha}\}\}$.

Now we consider the finite element approximation of above optimal control problems. We borrow some notations from \cite{Demlow}. To begin with, let $\mathcal{T}_0$ be a conforming and quasi-uniform partition of $\bar\Omega$ into disjoint $d$-simplices. Each element in $\mathcal{T}_0$ is assumed to be shape regular in the usual sense (see \cite{Ciarlet}). We denote the set of all conforming descendants $\mathcal{T}$ of $\mathcal{T}_0$ by $\mathbb{T}$, which can be generated through uniform or local refinements by newest vertex bisection algorithm. The simplices in any of those partitions are uniformly shape regular depending only on the shape regularity parameters of $\mathcal{T}_0$ and the dimension $d$, we refer to \cite[Appendix A]{Demlow} for more details on how to generate such kind of partitions. For each $\tilde {\mathcal{T}},\mathcal{T}\in \mathbb{T}$, we write $\mathcal{T}\subset \tilde{\mathcal{T}}$ when $\tilde{\mathcal{T}}$ is a refinement of $\mathcal{T}$.

Let $\mathcal{N}_{\mathcal{T}}$ and $\mathcal{E}_{\mathcal{T}}$ be the sets of vertices and interior edges or faces of $\mathcal{T}$. We also denote $\omega_T$ or $\tilde{\omega}_{T}$ the patches of elements sharing a vertex or a facet with $T$. We denote $h_T=|T|^{1\over d}$ for each $T\in \mathcal{T}\in \mathbb{T}$ with $|T|$ the volume of $T$. Since the Galerkin approximation is not the best approximation in $L^2$-norm, we need to impose some restrictions on the mesh for the convergence analysis of AFEM in $L^2$-norm . Following the ideas of \cite{Demlow} we define the continuous and piecewise linear mesh function $h_{\mathcal{T}}$, such that for any $z\in \mathcal{N}_{\mathcal{T}}$, $h_{\mathcal{T}}$ is defined as the average of the $h_{T'}$ over all $T'\in \mathcal{T}$ where $z\in T'$. Then for some constants $c_{\mathbb{T}}$ and $C_{\mathbb{T}}$ there holds
\begin{eqnarray}
c_{\mathbb{T}}h_T\leq h_{\mathcal{T}}|_{T}\leq C_{\mathbb{T}}h_T,\quad \forall T\in \mathcal{T}, \mathcal{T}\in \mathbb{T}.\label{shape_regular}
\end{eqnarray}
In view of the uniform shape regularity of $\mathbb{T}$ there exists another constant $\hat C_{\mathbb{T}}$ such that
\begin{eqnarray}
\|\nabla h_{\mathcal{T}}\|_{0.\infty,\Omega}\leq \hat C_{\mathbb{T}}\quad \forall \mathcal{T}\in \mathbb{T}.\nonumber
\end{eqnarray}
Throughout the paper we assume that the partition $\mathcal{T}$ is sufficiently graded, i.e., for some sufficiently small but fixed constant $\mu>0$, the mesh function $h_\mathcal{T}$ satisfies 
 \begin{eqnarray}
\|\nabla h_{\mathcal{T}}\|_{0,\infty,\Omega}\leq \mu,\label{mesh_grading}
\end{eqnarray}
and in addition, (\ref{shape_regular}) holds for some constants $c_{\mathbb{T}}$ and $C_{\mathbb{T}}$ that are independent of $\mu$. We refer to Appendix A in \cite{Demlow} on how to generate a class of sufficiently mildly graded partitions $\mathcal{T}\in \mathbb{T}$ for given  parameter $\mu$ such that the mesh function $h_\mathcal{T}$ satisfies (\ref{shape_regular}) and (\ref{mesh_grading}). Given a $\mu>0$ we denote the class of such  partitions by $\mathbb{T}_\mu$.

Associated with $\mathcal{T}\in \mathbb{T}$ we define the continuous and piecewise linear finite element  space $V_{\mathcal{T}}\subset H_0^1(\Omega)$. Let $\Pi_{\mathcal{T}}:C(\bar\Omega)\rightarrow V_{\mathcal{T}}$ be the standard Lagrange interpolation operator. We define a discrete control-to-state mapping as $S_\mathcal{T}:L^2(\Omega)\rightarrow V_{\mathcal{T}}$ such that $y_\mathcal{T}(u)=S_\mathcal{T}(f+u)$ and 
\begin{eqnarray}
a(y_\mathcal{T}(u),v_\mathcal{T})=(f+u,v_\mathcal{T})\quad \forall v_\mathcal{T}\in V_{\mathcal{T}}.\label{OCP_state_h}
\end{eqnarray}
Also we denote $y_\mathcal{T}(u)=S_\mathcal{T}u$ for simplicity. Then we can formulate a reduced discrete optimization problem $\min\limits_{u_\mathcal{T}\in U_{ad}}\ \hat J(u_\mathcal{T}):=J(S_\mathcal{T}u_\mathcal{T},u_\mathcal{T})$ where we use the variational control discretization proposed by Hinze (\cite{Hinze05COAP}). By standard arguments (\cite{Lions}) we can also prove that the above discrete optimization problem admits a unique solution. Moreover, we can derive the following discrete first order optimality condition: 
\begin{eqnarray}
(\alpha u_\mathcal{T}+p_\mathcal{T},v_\mathcal{T}-u_\mathcal{T})\geq 0\quad \forall v_\mathcal{T}\in U_{ad},\label{OCP_opt_h}
\end{eqnarray}
where the discrete adjoint state $p_\mathcal{T}\in V_\mathcal{T}$ satisfies
\begin{eqnarray}
a( v_\mathcal{T},p_\mathcal{T})=(y_\mathcal{T}-y_d,v_\mathcal{T})\quad \forall v_\mathcal{T}\in V_\mathcal{T}\label{OCP_adjoint_h}
\end{eqnarray}
with $y_\mathcal{T}=S_\mathcal{T}u_\mathcal{T}$. It is clear that $p_\mathcal{T}=S_\mathcal{T}^*(y_\mathcal{T}-y_d)$ where $S_\mathcal{T}^*$ is the adjoint of $S_\mathcal{T}$. Similarly, we have $u_\mathcal{T}=P_{[a,b]}(p_\mathcal{T})=\max\{a,\min\{b,-\frac{p_\mathcal{T}}{\alpha}\}\}$, which is not  generally a finite element function.

Now we will give some notations for the following purpose. For each element $T\in \mathcal{T}$ we define the  local error indicators $ \eta_{\mathcal{T},y}(u_\mathcal{T},y_\mathcal{T},T)$ contributed to the state equation and $ \eta_{\mathcal{T},p}(y_\mathcal{T},p_\mathcal{T},T)$ contributed to the adjoint state equation by
\begin{eqnarray}
\eta_{\mathcal{T},y}^2(u_\mathcal{T},y_\mathcal{T},T):=h_T^4\|f+u_\mathcal{T}-\mathcal{L}y_{\mathcal{T}}\|_{0,T}^2+\sum\limits_{E\in\mathcal{E}_\mathcal{T},E\subset\partial T}h_E^{3}\|[A\nabla y_\mathcal{T}]_E\cdot n_E\|_{0,E}^2,\label{local_state}\\
\eta_{\mathcal{T},p}^2(y_\mathcal{T},p_\mathcal{T},T):=h_T^4\|y_\mathcal{T}-y_d-\mathcal{L}^*p_{\mathcal{T}}\|_{0,T}^2+\sum\limits_{E\in\mathcal{E}_\mathcal{T},E\subset\partial T}h_E^{3}\|[A^*\nabla p_\mathcal{T}]_E\cdot n_E\|_{0,E}^2,\label{local_adjoint}
\end{eqnarray}
where $[A\nabla y_\mathcal{T}]_E\cdot n_E$ denotes the jump of $A\nabla y_\mathcal{T}$ across the common side $E$ of elements $T^+$ and $T^-$, $n_E$ denotes the outward normal oriented to $T^-$.  We also define the local error estimator for the optimal control problem
\begin{eqnarray}
\eta_{\mathcal{T}}^2(T):= \eta_{\mathcal{T},y}^2(u_\mathcal{T},y_\mathcal{T},T)+\eta_{\mathcal{T},p}^2(y_\mathcal{T},p_\mathcal{T},T).\label{local_OCP}
\end{eqnarray}
Then on a subset $\omega\subset \Omega$, we define the error estimator $\eta_{\mathcal{T},y}(u_\mathcal{T},y_\mathcal{T},\omega)$ by
\begin{eqnarray}
\eta_{\mathcal{T},y}^2(u_\mathcal{T},y_\mathcal{T},\omega):=\sum\limits_{T\in\mathcal{T},T\subset\omega}\eta_{\mathcal{T},y}^2(u_\mathcal{T},y_\mathcal{T},T).\label{estimator_elliptic}
\end{eqnarray}
Thus, $\eta_{\mathcal{T},y}(u_\mathcal{T},y_\mathcal{T},\mathcal{T})$ constitutes the error estimator for the state equation on $\Omega$ with respect to the partition $\mathcal{T}$. The similar definition applies to the error estimators $\eta_{\mathcal{T},p}(y_\mathcal{T},p_\mathcal{T},\mathcal{T})$ for the adjoint state equation and $\eta_{\mathcal{T}}(\mathcal{T})$ for the optimal control problem.

For $f\in L^2(\Omega)$ we also need to define the data oscillation as (see \cite{Morin,Nochetto})
\begin{eqnarray}
\mbox{osc}(f,T):=\|h_T^2(f-\bar f_T)\|_{0,T},\quad \mbox{osc}(f,\mathcal{T}):=\Big(\sum\limits_{T\in\mathcal{T}}\mbox{osc}^2(f,T)\Big)^{1\over 2},\label{data_osc}
\end{eqnarray}
where $\bar f_T$ denotes the $L^2$-projection of $f$ onto piecewise constant space on $T$. It is easy to see that 
\begin{eqnarray}
\mbox{osc}(f_1+f_2,\mathcal{T})\leq \mbox{osc}(f_1,\mathcal{T})+\mbox{osc}(f_2,\mathcal{T}),\quad \forall f_1,f_2\in L^2(\Omega).\label{osc_linear}
\end{eqnarray}


To derive error estimates in $L^2$-norm we need the following assumption:
\begin{Assumption}\label{Ass:1}
 Assume that $\Omega$ is convex so that for each $f+u\in L^2(\Omega)$ problem (\ref{OCP_state}) admits a unique solution $y=Su\in H^2(\Omega)\cap H_0^1(\Omega)$ and 
\begin{eqnarray}
\|y\|_{2,\Omega}\leq C_{\rm reg}\|f+u\|_{0,\Omega}.\label{Poisson_reg}
\end{eqnarray}
The assumption is also valid for the adjoint equation, i.e., for $S^*$.
\end{Assumption}

With above preparations now we are in the position to derive a posteriori error estimates for both the control, the state and adjoint state variables. 
\begin{Theorem}\label{Thm:1}
Let $(u,y,p)\in U_{ad}\times H_0^1(\Omega)\times H_0^1(\Omega)$ be the solution of optimal control problem (\ref{OCP})-(\ref{OCP_state}) and $(u_\mathcal{T},y_\mathcal{T},p_\mathcal{T})\in U_{ad}\times V_\mathcal{T}\times V_\mathcal{T}$ be the solution of the discrete problem (\ref{OCP_state_h})-(\ref{OCP_adjoint_h}). Then we have the a posteriori upper bound
\begin{eqnarray}
\|u-u_\mathcal{T}\|_{0,\Omega}+\|y-y_\mathcal{T}\|_{0,\Omega}+\|p-p_\mathcal{T}\|_{0,\Omega}\leq C_1\eta_\mathcal{T}(\mathcal{T})\label{OCP_error_upper}
\end{eqnarray}
and the global lower bound
\begin{eqnarray}
\eta_\mathcal{T}(\mathcal{T})\leq C_2(\|u-u_\mathcal{T}\|_{0,\Omega}+\|y-y_\mathcal{T}\|_{0,\Omega}+\|p-p_\mathcal{T}\|_{0,\Omega}+{\rm osc}_\mathcal{T}),\label{OCP_error_lower}
\end{eqnarray}
where $C_1$, $C_2$ only depend on the shape regularity of $\mathcal{T}$ and  the data oscillation ${\rm osc}_\mathcal{T}$ is defined as 
\begin{eqnarray}
{\rm osc}_\mathcal{T}^2:={\rm osc}^2(f+u_\mathcal{T}-\mathcal{L}y_{\mathcal{T}},\mathcal{T})+{\rm osc}^2(y_\mathcal{T}-y_d-\mathcal{L}^*p_{\mathcal{T}},\mathcal{T}).\nonumber
\end{eqnarray}
\end{Theorem}
\begin{proof}
Setting $v= u_\mathcal{T}\in U_{ad}$ in (\ref{OCP_opt}) and $v_\mathcal{T} = u\in U_{ad}$ in (\ref{OCP_opt_h}) and noticing that $p=S^*(Su-y_d)$, $p_\mathcal{T}=S^*_\mathcal{T}(S_\mathcal{T}u_\mathcal{T}-y_d)$,  we are then led to
\begin{eqnarray}
(\alpha u+S^*(Su-y_d),u_\mathcal{T}-u)\geq 0, \label{u_to_p1}\\
(\alpha u_\mathcal{T}+S^*_\mathcal{T}(S_\mathcal{T}u_\mathcal{T}-y_d),u-u_\mathcal{T})\geq 0. \label{u_to_p2}
\end{eqnarray}
Adding the above two inequalities, we conclude from (\ref{OCP_state_weak}) and (\ref{OCP_adjoint_weak})  that
\begin{eqnarray}
&&\alpha\|u- u_\mathcal{T}\|^2_{0,\Omega}\leq (S^*_\mathcal{T}(S_\mathcal{T}u_\mathcal{T}-y_d)-S^*(Su-y_d),u- u_\mathcal{T})\nonumber\\
&=& (S^*_\mathcal{T}(S_\mathcal{T} u_\mathcal{T}-y_d)-S^*(S_\mathcal{T}u_\mathcal{T}-y_d),u- u_\mathcal{T})+(S^*(S_\mathcal{T}u_\mathcal{T}-y_d)-S^*(Su-y_d),u-u_\mathcal{T})\nonumber\\
&=&(S^*_\mathcal{T}(S_\mathcal{T}u_\mathcal{T}-y_d)-S^*(S_\mathcal{T}u_\mathcal{T}-y_d),u-u_\mathcal{T})+(S_\mathcal{T}u_\mathcal{T}-Su,Su-Su_\mathcal{T})\nonumber\\
&=&(S^*_\mathcal{T}(S_\mathcal{T} u_\mathcal{T}-y_d)-S^*(S_\mathcal{T}u_\mathcal{T}-y_d),u-u_\mathcal{T})+(S_\mathcal{T}u_\mathcal{T}-Su,Su-S_\mathcal{T}u_\mathcal{T})\nonumber\\
&&+(S_\mathcal{T} u_\mathcal{T}-Su,S_\mathcal{T} u_\mathcal{T}-Su_\mathcal{T}).\label{u_rep}
\end{eqnarray}
It follows from Young's inequality that
\begin{eqnarray}
&&\alpha\|u-u_\mathcal{T}\|^2_{0,\Omega}+\|y-y_\mathcal{T}\|^2_{0,\Omega}\nonumber\\
&\leq&  C\|Su_\mathcal{T}-S_\mathcal{T} u_\mathcal{T}\|^2_{0,\Omega}+C\|S^*(S_\mathcal{T}u_\mathcal{T}-y_d)-S^*_\mathcal{T}(S_\mathcal{T}u_\mathcal{T}-y_d)\|_{0,\Omega}^2,\label{u_error}
\end{eqnarray}
where we used the fact that $y=Su$ and $y_\mathcal{T}=S_\mathcal{T} u_\mathcal{T}$. Moreover, from (\ref{OCP_state_weak}), (\ref{OCP_adjoint_weak}) and the triangle inequality we have
\begin{eqnarray}
\|p- p_\mathcal{T}\|_{0,\Omega}&\leq& \|p-S^*(S_\mathcal{T}u_\mathcal{T}-y_d)\|_{0,\Omega}+\|S^*(S_\mathcal{T}u_\mathcal{T}-y_d)- p_\mathcal{T}\|_{0,\Omega}\nonumber\\
&\leq& C\|Su-S_\mathcal{T}u_\mathcal{T}\|_{0,\Omega}+\|S^*(S_\mathcal{T}u_\mathcal{T}-y_d)-p_\mathcal{T}\|_{0,\Omega}\nonumber\\
&\leq& C\|u-u_\mathcal{T}\|_{0,\Omega}+C\|S u_\mathcal{T}-S_\mathcal{T}u_\mathcal{T}\|_{0,\Omega}+\|S^*(S_\mathcal{T}u_\mathcal{T}-y_d)-p_\mathcal{T}\|_{0,\Omega}.\nonumber
\end{eqnarray}
Combining the above estimates we are led to 
 \begin{eqnarray}
&&\alpha\|u-u_\mathcal{T}\|^2_{0,\Omega}+\|y-y_\mathcal{T}\|^2_{0,\Omega}+\|p- p_\mathcal{T}\|_{0,\Omega}^2\nonumber\\
 &\leq&C\|Su_\mathcal{T}-S_\mathcal{T} u_\mathcal{T}\|^2_{0,\Omega}+C\|S^*(S_\mathcal{T}u_\mathcal{T}-y_d)-S^*_\mathcal{T}(S_\mathcal{T}u_\mathcal{T}-y_d)\|_{0,\Omega}^2.
 \end{eqnarray}
Note that $S_\mathcal{T} u_\mathcal{T}$ and $S^*_\mathcal{T}(S_\mathcal{T}u_\mathcal{T}-y_d)$ are the standard finite element approximations of $Su_\mathcal{T}$ and $S^*(S_\mathcal{T}u_\mathcal{T}-y_d)$ in finite element space $V_\mathcal{T}$, respectively. Under  Assumption \ref{Ass:1} we can derive from standard a posteriori error estimate for elliptic equation under $L^2$-norm that (see \cite{Verfuth} for more details)
\begin{eqnarray}
\|Su_\mathcal{T}-S_\mathcal{T} u_\mathcal{T}\|_{0,\Omega}\leq C\eta_{\mathcal{T},y}(u_\mathcal{T},y_\mathcal{T},\mathcal{T}),\\
\|S^*(S_\mathcal{T}u_\mathcal{T}-y_d)-S^*_\mathcal{T}(S_\mathcal{T}u_\mathcal{T}-y_d)\|_{0,\Omega}\leq C\eta_{\mathcal{T},p}(y_\mathcal{T},p_\mathcal{T},\mathcal{T}).
\end{eqnarray}
Therefore, combining the above results we are able to derive the upper bound with the constant $C_1$ independent of the mesh size $h_\mathcal{T}$. 

Now we prove the lower bound. Note that
\begin{eqnarray}
\|S u_\mathcal{T}-S_\mathcal{T}u_\mathcal{T}\|_{0,\Omega}&\leq&\|Su_\mathcal{T}-Su\|_{0,\Omega}+\|Su-S_\mathcal{T}u_\mathcal{T}\|_{0,\Omega}\nonumber\\
&\leq&C\| u- u_\mathcal{T}\|_{0,\Omega}+\| y- y_\mathcal{T}\|_{0,\Omega}.\label{y_error}
\end{eqnarray}
Similarly, we can derive that
\begin{eqnarray}
&&\|S^*(S_\mathcal{T}u_\mathcal{T}-y_d)-S^*_\mathcal{T}(S_\mathcal{T}u_\mathcal{T}-y_d)\|_{0,\Omega}\nonumber\\
&\leq&\|S^*(S_\mathcal{T}u_\mathcal{T}-y_d)-S^*(Su-y_d)\|_{0,\Omega}+\|S^*(Su-y_d)-S^*_\mathcal{T}(S_\mathcal{T}u_\mathcal{T}-y_d)\|_{0,\Omega}\nonumber\\
&\leq&C\|y-y_\mathcal{T}\|_{0,\Omega}+\| p- p_\mathcal{T}\|_{0,\Omega}.\label{p_error}
\end{eqnarray}
Moreover, from standard lower bound error estimate for elliptic equation (see \cite{Verfuth} for more details)
 we have the following global lower bound
 \begin{eqnarray}
\eta_{\mathcal{T},y}(u_\mathcal{T},y_\mathcal{T},\mathcal{T})&\leq& C(\|Su_\mathcal{T}-S_\mathcal{T} u_\mathcal{T}\|_{0,\Omega}+\mbox{osc}(f+u_\mathcal{T}-\mathcal{L}y_{\mathcal{T}},\mathcal{T})),\\
\eta_{\mathcal{T},p}(y_\mathcal{T},p_\mathcal{T},\mathcal{T})&\leq& C(\|S^*(S_\mathcal{T}u_\mathcal{T}-y_d)-S^*_\mathcal{T}(S_\mathcal{T}u_\mathcal{T}-y_d)\|_{0,\Omega}\nonumber\\
&&+\mbox{osc}(y_\mathcal{T}-y_d-\mathcal{L}^*p_{\mathcal{T}},\mathcal{T})).
\end{eqnarray}
 Thus, we can conclude from the above estimates the lower bound with the constant $C_2$ independent of the mesh size $h_\mathcal{T}$. This completes the proof.
\end{proof}
\begin{Remark}
We remark that it is also possible to derive reliable a posteriori error estimates if $\Omega$ is a polygon but  not necessarily convex, by using the regularity of the dual problem in weighted Sobolev spaces when we derive an error estimator in $L^2$-norm for second order elliptic equation, see, e.g., \cite{Wihler}. However, for our convergence analysis we also need the efficiency of the error estimator which is still missing in the literature for unconvex domain. Therefore, we restrict ourselves to the convex case in this paper.
\end{Remark}

In the following we will present the adaptive algorithm for solving optimal control problems. Note that there are two error estimators $\eta_{\mathcal{T},y}(u_\mathcal{T},y_\mathcal{T},T)$ and $\eta_{\mathcal{T},p}(y_\mathcal{T},p_\mathcal{T},T)$ contributed to the state approximation and adjoint state approximation, respectively. We use the sum of the two estimators as our indicators for the marking strategy. The marking algorithm based on D\"{o}rfler's strategy for optimal control problems can be described as follows
\begin{Algorithm}\label{Alg:3.0}D\"{o}rfler's marking strategy for OCPs
\begin{enumerate}
\item Given a parameter $0<\theta<1$;

\item Construct a minimal subset ${\mathcal{M}}\subset \mathcal{T}$ such that
\begin{equation*}
\sum\limits_{T\in{\mathcal{M}}}\eta_{\mathcal{T}}^2(T)\geq \theta^2\eta_{\mathcal{T}}^2(\mathcal{T}).
\end{equation*}

\item Mark all the elements in ${\mathcal{M}}$.

\end{enumerate}
\end{Algorithm}

In the following we will frequently use the abbreviations $V_k$ for $V_{\mathcal{T}_{k}}$, $h_k$ for $h_{\mathcal{T}_{k}}$ and $v_{k}$ for $v_{\mathcal{T}_k}$, and the similar abbreviations for other variables and notations. Now we can present the adaptive finite element algorithm for the optimal
control problem  as follows. 
\begin{Algorithm}\label{Alg:3.1}
Adaptive finite element algorithm for OCPs:
\begin{enumerate}
\item Given an initial mesh $\mathcal{T}_{0}$ with mesh size $h_0$ and the associated finite element space $V_{0}$.

\item Set $k=0$ and solve the optimal control problem (\ref{OCP_state_h})-(\ref{OCP_adjoint_h}) to obtain 
$(u_{k},y_{k},p_{k})\in U_{ad}\times V_{k}\times V_{k}$.

\item Compute the local error indicator $\eta_{k}(T)$.

\item Construct ${ \mathcal{M}}_{k}\subset\mathcal{T}_{k}$ by the marking Algorithm \ref{Alg:3.0}.

\item Refine ${ \mathcal{M}}_{k}$ to get a new conforming mesh $\mathcal{T}_{{k+1}}$ by procedure REFINE using bisection algorithm.

\item Construct the finite element space $V_{{k+1}}$ and solve the optimal control problem (\ref{OCP_state_h})-(\ref{OCP_adjoint_h}) to obtain 
$(u_{{k+1}},y_{{k+1}},p_{{k+1}})\in U_{ad}\times V_{{k+1}}\times V_{{k+1}}$.

\item Set $k=k+1$ and go to Step (3).

\end{enumerate}
\end{Algorithm}

In step (5) of Algorithm \ref{Alg:3.1} we assume that each marked element in $\mathcal{M}_k$ is bisected $r\geq 1$ times to generate a new mesh $\mathcal{T}_{k+1}$, and additional elements are refined in the process to ensure that $\mathcal{T}_{k+1}$ is conforming. We remark that to ensure the mesh grading property (\ref{mesh_grading}) we have to additionally refine elements other than that of ${ \mathcal{M}}_{k}$. Demlow and Stevenson \cite{Demlow} showed that this can be done by inflating the number of simplices by not more than some fixed multiple which depends on the mesh grading parameter $\mu$ and can be achieved by the standard newest vertex bisection algorithm with necessary modifications, and the modification does not compromise the quasi-optimality of the resulting algorithm, we refer to Appendix A in \cite{Demlow} for more details.

\section{Convergence of AFEM for the optimal control problem in $L^2$-norm}
\setcounter{equation}{0}

In this section we intend to prove the contraction property of the $L^2$-norm errors of the control, the state and adjoint state with $L^2$-norm based AFEM. The proof relies on establishing certain equivalence property between the $L^2$-norm error and the weighted energy norm error for the state and adjoint state, and uses the convergence result of AFEM under energy norm.

\subsection{Connections between the weighted energy norm errors and $L^2$-norm errors}
In this subsection we show certain relationships between the energy norm and $L^2$-norm errors for the finite element approximations of optimal control problems by generalizing the resluts in \cite{Demlow}.

 At first, we show that the $L^2$-norm error of finite element approximation of elliptic equation can be bounded by the weighted energy norm as long as the mesh grading is sufficiently mild. 
\begin{Lemma}\label{Lemma:1}
For sufficiently small $\mu$, let $\mathcal{T}\in \mathbb{T}_\mu$. Under Assumption \ref{Ass:1} we have that for any $f,g\in L^2(\Omega)$, 
\begin{eqnarray}
\|Sf-S_\mathcal{T}f\|_{0,\Omega}\lesssim \||h_\mathcal{T}(Sf-S_\mathcal{T}f)\||,\label{3.1}\\
\|S^*g-S_\mathcal{T}^*g\|_{0,\Omega}\lesssim \||h_\mathcal{T}(S^*g-S_\mathcal{T}^*g)\||.
\end{eqnarray}
\end{Lemma}
\begin{proof}
We only prove (\ref{3.1}) by using the duality argument. In \cite[Proposition 3]{Demlow} the authors proved the result for the Laplace equation, here we extend the proof to a more general second order elliptic equation.  Let $\phi\in H^2(\Omega)\cap H_0^1(\Omega)$ be the unique solution of the auxiliary problem
\begin{eqnarray}
a(v,\phi)=(Sf-S_\mathcal{T}f,v)\quad\forall v\in H_0^1(\Omega).\nonumber
\end{eqnarray}
Then we have $\|\phi\|_{2,\Omega}\leq C\|Sf-S_\mathcal{T}f\|_{0,\Omega}$. Using Galerkin orthogonality and the interpolation error estimates we have
\begin{eqnarray}
\|Sf-S_\mathcal{T}f\|_{0,\Omega}^2&=&a(Sf-S_\mathcal{T}f,\phi)=a(Sf-S_\mathcal{T}f,\phi-\Pi_\mathcal{T}\phi)\nonumber\\
&=&(h_\mathcal{T}A\nabla(Sf-S_\mathcal{T}f),h_\mathcal{T}^{-1}\nabla(\phi-\Pi_\mathcal{T}\phi))+(a_0h_\mathcal{T}(Sf-S_\mathcal{T}f),h_\mathcal{T}^{-1}(\phi-\Pi_\mathcal{T}\phi))\nonumber\\
&\lesssim&(\|\nabla (h_\mathcal{T}(Sf-S_\mathcal{T}f))\|_{0,\Omega}+\|\nabla h_\mathcal{T}(Sf-S_\mathcal{T}f)\|_{0,\Omega})\|h_\mathcal{T}^{-1}\nabla(\phi-\Pi_\mathcal{T}\phi)\|_{0,\Omega}\nonumber\\
&&+\|h_\mathcal{T}(Sf-S_\mathcal{T}f)\|_{0,\Omega}\|h_\mathcal{T}^{-1}(\phi-\Pi_\mathcal{T}\phi)\|_{0,\Omega}\nonumber\\
&\lesssim&\||h_\mathcal{T}(Sf-S_\mathcal{T}f)\||\cdot\|Sf-S_\mathcal{T}f\|_{0,\Omega}+\mu\|Sf-S_\mathcal{T}f\|_{0,\Omega}^2,
\end{eqnarray}
where we used the properties (\ref{shape_regular}) and (\ref{mesh_grading}). Taking $\mu$ small enough we complete the proof.
\end{proof}
We remark that similar to the proof of Lemma \ref{Lemma:1} we can extend the results of \cite{Demlow} to a general second order elliptic equation with sufficiently smooth coefficients. Therefore, we will use some results of \cite{Demlow} without proof in the following analysis. 

With above result we can establish the connection between the $L^2$-norm errors and weighted energy norm errors for the solutions of optimal control problems.
\begin{Lemma}\label{Lemma:11}
Let $(u,y,p)\in U_{ad}\times H_0^1(\Omega)\times H_0^1(\Omega)$ be the solution of optimal control problem (\ref{OCP})-(\ref{OCP_state}) and $(u_\mathcal{T},y_\mathcal{T},p_\mathcal{T})\in U_{ad}\times V_\mathcal{T}\times V_\mathcal{T}$ be the solution of the discrete problem (\ref{OCP_state_h})-(\ref{OCP_adjoint_h}). For sufficiently small $\mu$, let $\mathcal{T}\in \mathbb{T}_\mu$. Under Assumption \ref{Ass:1} we have that
\begin{eqnarray}
\|u-u_\mathcal{T}\|_{0,\Omega}+\|y-y_\mathcal{T}\|_{0,\Omega}+\|p-p_\mathcal{T}\|_{0,\Omega}
\lesssim \||h_\mathcal{T}(y-y_\mathcal{T})\||+\||h_\mathcal{T}(p-p_\mathcal{T})\||
\end{eqnarray}
provided that $h_0\ll 1$.
\end{Lemma}
\begin{proof}
From the proof of Theorem \ref{Thm:1} we can conclude that
\begin{eqnarray}
&&\|u-u_\mathcal{T}\|_{0,\Omega}+\|y-y_\mathcal{T}\|_{0,\Omega}+\|p-p_\mathcal{T}\|_{0,\Omega}\nonumber\\
&\lesssim& \|Su_\mathcal{T}-y_\mathcal{T}\|_{0,\Omega}+\|S^*(y_\mathcal{T}-y_d)-p_\mathcal{T}\|_{0,\Omega}.\label{OCP_poste}
\end{eqnarray}
Recall that $y_\mathcal{T}$ and $p_\mathcal{T}$ are the standard finite element approximations of $Su_\mathcal{T}$ and $S^*(y_\mathcal{T}-y_d)$ in $V_\mathcal{T}$, respectively. Then it follows from Lemma \ref{Lemma:1} and the triangle inequality that
\begin{eqnarray}
&&\|u-u_\mathcal{T}\|_{0,\Omega}+\|y-y_\mathcal{T}\|_{0,\Omega}+\|p-p_\mathcal{T}\|_{0,\Omega}\nonumber\\
&\lesssim&\||h_\mathcal{T}(Su_\mathcal{T}-y_\mathcal{T})\||+\||h_\mathcal{T}(S^*(y_\mathcal{T}-y_d)-p_\mathcal{T})\||\nonumber\\
&\lesssim&\||h_\mathcal{T}(y-y_\mathcal{T})\||+\||h_\mathcal{T}(p-p_\mathcal{T})\||\nonumber\\
&&+\||h_\mathcal{T}(y-Su_\mathcal{T})\||+\||h_\mathcal{T}(p-S^*(y_\mathcal{T}-y_d))\||.
\end{eqnarray}
After elementary calculation we have
\begin{eqnarray}
\||h_\mathcal{T}(y-Su_\mathcal{T})\||^2&=&(A\nabla (h_\mathcal{T}(y-Su_\mathcal{T})),\nabla (h_\mathcal{T}(y-Su_\mathcal{T})))+(a_0h_\mathcal{T}(y-Su_\mathcal{T}),h_\mathcal{T}(y-Su_\mathcal{T}))\nonumber\\
&=&(A\nabla (y-Su_\mathcal{T}),\nabla (h_\mathcal{T}^2(y-Su_\mathcal{T})))+(A\nabla h_\mathcal{T}(y-Su_\mathcal{T}),\nabla h_\mathcal{T}(y-Su_\mathcal{T}))\nonumber\\
&&+(a_0h_\mathcal{T}(y-Su_\mathcal{T}),h_\mathcal{T}(y-Su_\mathcal{T}))\nonumber\\
&\lesssim&\| h_\mathcal{T}\|_{0,\infty,\Omega}^2\||y-Su_\mathcal{T}\||^2+ \|\nabla h_\mathcal{T}\|_{0,\infty,\Omega}^2\|y-Su_\mathcal{T}\|_{0,\Omega}^2\nonumber\\
&&+\| h_\mathcal{T}\|_{0,\infty,\Omega} \|\nabla h_\mathcal{T}\|_{0,\infty,\Omega}\||y-Su_\mathcal{T}\||\cdot\|y-Su_\mathcal{T}\|_{0,\Omega}.
\end{eqnarray}
Note that $\| h_\mathcal{T}\|_{0,\infty,\Omega}\leq h_0$ and $\|\nabla h_\mathcal{T}\|_{0,\infty,\Omega}\leq \mu$. This combining with the stability of elliptic equation gives 
\begin{eqnarray}
\||h_\mathcal{T}(y-Su_\mathcal{T})\||\lesssim (h_0+\mu)\|u-u_\mathcal{T}\|_{0,\Omega}.
\end{eqnarray}
Similarly, we can prove that
\begin{eqnarray}
\||h_\mathcal{T}(p-S^*(y_\mathcal{T}-y_d))\||&\lesssim& (h_0+\mu)\|y-y_\mathcal{T}\|_{0,\Omega}\nonumber\\
&\lesssim&(h_0+\mu)(\|u-u_\mathcal{T}\|_{0,\Omega}+\||h_\mathcal{T}(y-y_\mathcal{T})\||).
\end{eqnarray}
Combining the above results we complete the proof if $h_0\ll 1$ and $\mu$ is sufficiently small.
\end{proof}

Now we can derive a posteriori error estimate for the optimal control problem under weighted energy norm.
\begin{Lemma}\label{Lemma:12}
 Let Assumption \ref{Ass:1} be valid. For any $\mathcal{T}\in \mathbb{T}$, there exists a constant $C_3$ independent of the mesh size of $\mathcal{T}$ such that
\begin{eqnarray}
\||h_\mathcal{T}(y-y_\mathcal{T})\||^2+\||h_\mathcal{T}(p-p_\mathcal{T})\||^2\leq C_3\eta^2_{\mathcal{T}}(\mathcal{T}).
\end{eqnarray}
\end{Lemma}
\begin{proof}
By using Proposition 4 in \cite{Demlow} and noticing that $S_\mathcal{T} u_\mathcal{T}$ and $S^*_\mathcal{T}(S_\mathcal{T}u_\mathcal{T}-y_d)$ are the standard finite element approximations of $Su_\mathcal{T}$ and $S^*(S_\mathcal{T}u_\mathcal{T}-y_d)$ in finite element space $V_\mathcal{T}$, we have the following a posteriori upper bounds for $\||h_\mathcal{T}(Su_\mathcal{T}-S_\mathcal{T}u_{\mathcal{T}})\||$ and $\||h_\mathcal{T}(S^*(S_\mathcal{T}u_\mathcal{T}-y_d)-S^*_\mathcal{T}(S_\mathcal{T}u_\mathcal{T}-y_d))\||$:
\begin{eqnarray}
\||h_\mathcal{T}(Su_\mathcal{T}-S_\mathcal{T}u_{\mathcal{T}})\||\lesssim\eta_{\mathcal{T},y}(u_\mathcal{T},y_\mathcal{T},\mathcal{T}),\label{Lemma3.1_1}\\
\||h_\mathcal{T}(S^*(S_\mathcal{T}u_\mathcal{T}-y_d)-S^*_\mathcal{T}(S_\mathcal{T}u_\mathcal{T}-y_d))\||\lesssim\eta_{\mathcal{T},p}(y_\mathcal{T},p_\mathcal{T},\mathcal{T}).\label{Lemma3.1_2}
\end{eqnarray}
From the triangle inequality now it suffices to estimate $\||h_\mathcal{T}(y-Su_\mathcal{T})\||$ and $\||h_\mathcal{T}(p-S^*(S_\mathcal{T}u_\mathcal{T}-y_d))\||$. We use the abbreviation $e_\mathcal{T}=y-Su_\mathcal{T}$ and note that
\begin{eqnarray}
\||h_\mathcal{T}(y-Su_\mathcal{T})\||^2&=&(A\nabla (h_\mathcal{T}e_\mathcal{T}),\nabla(h_\mathcal{T}e_\mathcal{T}))+(a_0h_\mathcal{T}e_\mathcal{T},h_\mathcal{T}e_\mathcal{T})\nonumber\\
&=&(A\nabla h_\mathcal{T}e_\mathcal{T},\nabla h_\mathcal{T}e_\mathcal{T})+(Ah_\mathcal{T}\nabla e_\mathcal{T},h_\mathcal{T}\nabla e_\mathcal{T})+2(A\nabla h_\mathcal{T}e_\mathcal{T},h_\mathcal{T}\nabla e_\mathcal{T})\nonumber\\
&&+(a_0h_\mathcal{T}e_\mathcal{T},h_\mathcal{T}e_\mathcal{T})\nonumber\\
&\lesssim&\|\nabla h_\mathcal{T}\|_{0,\infty,\Omega}^2\|e_\mathcal{T}\|_{0,\Omega}^2+\|h_\mathcal{T}\|_{0,\infty,\Omega}^2\||e_\mathcal{T}\||^2\nonumber\\
&&+2\|\nabla h_\mathcal{T}\|_{0,\infty,\Omega}\|h_\mathcal{T}\|_{0,\infty,\Omega}\|e_\mathcal{T}\|_{0,\Omega}\||e_\mathcal{T}\||\nonumber\\
&\lesssim& \|u-u_\mathcal{T}\|_{0,\Omega}^2,\label{energy_1}
\end{eqnarray}
where we used the fact that $\| h_\mathcal{T}\|_{0,\infty,\Omega}\lesssim 1$, $\|\nabla h_\mathcal{T}\|_{0,\infty,\Omega}\lesssim 1$ and the stability result for elliptic boundary value problem. Similarly, we can conclude from Lemma \ref{Lemma:1} and (\ref{Lemma3.1_1}) that
\begin{eqnarray}
\||h_\mathcal{T}(p-S^*(S_\mathcal{T}u_\mathcal{T}-y_d))\||^2&\lesssim& \|y-y_\mathcal{T}\|_{0,\Omega}^2\nonumber\\
&\lesssim&\|y-Su_\mathcal{T}\|_{0,\Omega}^2+\||h_\mathcal{T}(Su_\mathcal{T}-S_\mathcal{T}u_{\mathcal{T}})\||^2\nonumber\\
&\lesssim&\|u-u_\mathcal{T}\|_{0,\Omega}^2+\eta^2_{\mathcal{T},y}(u_\mathcal{T},y_\mathcal{T},\mathcal{T}).\label{energy_2}
\end{eqnarray}
Combining Theorem \ref{Thm:1}, (\ref{Lemma3.1_1})-(\ref{energy_2}) and the triangle inequality we finish the proof.
\end{proof}

It follows from Theorem \ref{Thm:1} and Lemma \ref{Lemma:12} that
\begin{eqnarray}
\||h_\mathcal{T}(y-y_\mathcal{T})\||^2+\||h_\mathcal{T}(p-p_\mathcal{T})\||^2\lesssim \|u-u_\mathcal{T}\|_{0,\Omega}^2+\|y-y_\mathcal{T}\|_{0,\Omega}^2+\|p-p_\mathcal{T}\|_{0,\Omega}^2+{\rm osc}_{\mathcal{T}}^2.\nonumber
\end{eqnarray}
The following stability results for the error estimators are  direct consequences of \cite[Lemma 1]{Demlow}, see also \cite[Proposition 3.3]{Cascon}.
\begin{Lemma}\label{Lemma:3}
For any $\mathcal{T}\in \mathbb{T}$, let $u_\mathcal{T},\tilde{u}_\mathcal{T}\in U_{ad}$, $y_\mathcal{T},\tilde{y}_\mathcal{T},p_\mathcal{T},\tilde{p}_\mathcal{T}\in V_\mathcal{T}$. Under Assumption \ref{Ass:1} we have that 
\begin{eqnarray}
|\eta_{\mathcal{T},y}(u_\mathcal{T},y_\mathcal{T},T)-\eta_{\mathcal{T},y}(\tilde u_\mathcal{T},\tilde y_\mathcal{T},T)|&\lesssim& \||h_\mathcal{T}(y_\mathcal{T}-\tilde y_\mathcal{T})\||_{\tilde{\omega}_T}+h_T^2\|u_\mathcal{T}-\tilde u_\mathcal{T}\|_{0,T}\nonumber\\
&&+\|\nabla h_\mathcal{T}\|_{0,\infty,\tilde{\omega}_T}\|y_\mathcal{T}-\tilde y_\mathcal{T}\|_{0,\tilde{\omega}_T},\label{estimatory_pert}\\
|\eta_{\mathcal{T},p}(y_\mathcal{T},p_\mathcal{T},T)-\eta_{\mathcal{T},p}(\tilde y_\mathcal{T},\tilde p_\mathcal{T},T)|&\lesssim& \||h_\mathcal{T}(p_\mathcal{T}-\tilde p_\mathcal{T})\||_{\tilde{\omega}_T}+h_T^2\|y_\mathcal{T}-\tilde y_\mathcal{T}\|_{0,T}\nonumber\\
&&+\|\nabla h_\mathcal{T}\|_{0,\infty,\tilde{\omega}_T}\|p_\mathcal{T}-\tilde p_\mathcal{T}\|_{0,\tilde{\omega}_T}.\label{estimatorp_pert}
\end{eqnarray}
\end{Lemma}
\begin{proof}
We only prove (\ref{estimatory_pert}), the proof of (\ref{estimatorp_pert}) is very similar and we will omit it. Note that
\begin{eqnarray}
\eta_{\mathcal{T},y}(u_\mathcal{T},y_\mathcal{T},T)&\leq& \eta_{\mathcal{T},y}(\tilde u_\mathcal{T},\tilde y_\mathcal{T},T)+(h_T^4\|u_\mathcal{T}-\tilde u_\mathcal{T}-\mathcal{L}(y_\mathcal{T}-\tilde y_\mathcal{T})\|_{0,T}^2\nonumber\\
&&+\sum\limits_{E\in \mathcal{E}_\mathcal{T},E\subset\partial T}h_E^{3}\|[A\nabla (y_\mathcal{T}-\tilde y_\mathcal{T})]_E\cdot n_E\|_{0,E}^2)^{1\over 2}.\nonumber
\end{eqnarray}
Noting that $y_\mathcal{T}$ and $\tilde y_\mathcal{T}$  are linear polynomials on $T$, it is easy to verify that
\begin{eqnarray}
h_T^2\|\mathcal{L}(y_\mathcal{T}-\tilde y_\mathcal{T})\|_{0,T}&=&h_T^2(\|-\mbox{div} A\cdot\nabla (y_\mathcal{T}-\tilde y_\mathcal{T})+a_0(y_\mathcal{T}-\tilde y_\mathcal{T})\|_{0,T})\nonumber\\
&\lesssim&\||h_\mathcal{T} (y_\mathcal{T}-\tilde y_\mathcal{T})\||_{T}.
\end{eqnarray}
We recall the trace inequality: for any $T\in\mathcal{T}$ and $v_\mathcal{T}\in V_\mathcal{T}$ there holds
\begin{eqnarray}
\|\nabla v_\mathcal{T}\|_{0,\partial T}\lesssim h_T^{-{1\over 2}}\|\nabla v_\mathcal{T}\|_{0,T}+h_T^{{1\over 2}}\|\nabla^2 v_\mathcal{T}\|_{0,T}.\nonumber
\end{eqnarray}
The second term in above inequality vanishes because $v_\mathcal{T}$ is linear polynomial on $T$. Thus, for the edge $E=T\cap T'$ we have
\begin{eqnarray}
h_E^{3\over 2}\|[A\nabla (y_\mathcal{T}-\tilde y_\mathcal{T})]_E\cdot n_E\|_{0,E}&\lesssim &h_E^{3\over 2}(\|\nabla (y_\mathcal{T}-\tilde y_\mathcal{T})_T\|_{0,E}+\|\nabla (y_\mathcal{T}-\tilde y_\mathcal{T})_{T'}\|_{0,E})\nonumber\\
&\lesssim &h_T(\|\nabla (y_\mathcal{T}-\tilde y_\mathcal{T})\|_{0,T}+\|\nabla (y_\mathcal{T}-\tilde y_\mathcal{T})\|_{0,T'})\nonumber\\
&\lesssim&\||h_\mathcal{T} (y_\mathcal{T}-\tilde y_\mathcal{T})\||_{T\cup T'}+\|\nabla h_\mathcal{T} (y_\mathcal{T}-\tilde y_\mathcal{T})\|_{0,T\cup T'},
\end{eqnarray}
where we used the property (\ref{shape_regular}). Summing over the edges of $T$ we complete the proof.
\end{proof}


The following lemma presents a quasi-orthogonality result for the solution of elliptic boundary value problem.
\begin{Lemma}(\cite[Lemma 3]{Demlow})\label{Lemma:5}
For any $\epsilon>0$, $\mathcal{T}\in \mathbb{T}$, $\tilde y_\mathcal{T},\tilde p_\mathcal{T}\in V_{\mathcal{T}}$ and $f,g\in L^2(\Omega)$, it holds that
\begin{eqnarray}
&&\||h_\mathcal{T}(Sf-S_\mathcal{T}f)\||^2+\||h_\mathcal{T}(S_\mathcal{T}f-\tilde y_\mathcal{T})\||^2-(1+\epsilon)\||h_\mathcal{T}(Sf-\tilde y_\mathcal{T})\||^2\nonumber\\
&\leq& \epsilon^{-1}\|\nabla h_\mathcal{T}\|_{0,\infty,\Omega}^2(\|Sf-S_\mathcal{T}f\|_{0,\Omega}^2+\|Sf-\tilde y_\mathcal{T}\|_{0,\Omega}^2),\\
&&\||h_\mathcal{T}(S^*g-S^*_\mathcal{T}g)\||^2+\||h_\mathcal{T}(S^*_\mathcal{T}g-\tilde p_\mathcal{T})\||^2-(1+\epsilon)\||h_\mathcal{T}(S^*g-\tilde p_\mathcal{T})\||^2\nonumber\\
&\leq& \epsilon^{-1}\|\nabla h_\mathcal{T}\|_{0,\infty,\Omega}^2(\|S^*g-S^*_\mathcal{T}g\|_{0,\Omega}^2+\|S^*g-\tilde p_\mathcal{T}\|_{0,\Omega}^2).
\end{eqnarray}
\end{Lemma}

We also recall the following inequalities which are opposed to Lemma \ref{Lemma:1}.
\begin{Lemma}(\cite[Lemma 4]{Demlow})\label{Lemma:6}
For any $\mathcal{T}\in \mathbb{T}$ and $f,g \in L^2(\Omega)$ there holds
\begin{eqnarray}
\||h_\mathcal{T}(Sf-\tilde y_\mathcal{T})\||\lesssim \|Sf-\tilde y_\mathcal{T}\|_{0,\Omega}+{\rm osc}(f-\mathcal{L}\tilde y_\mathcal{T},\mathcal{T}),\quad \forall \tilde y_\mathcal{T}\in V_\mathcal{T},\\
\||h_\mathcal{T}(S^*g-\tilde p_\mathcal{T})\||\lesssim \|S^*g-\tilde p_\mathcal{T}\|_{0,\Omega}+{\rm osc}(g-\mathcal{L}^*\tilde p_\mathcal{T},\mathcal{T}),\quad \forall \tilde p_\mathcal{T}\in V_\mathcal{T}.
\end{eqnarray}
\end{Lemma}

As a final preliminary result we show that the $L^2$ norm errors of the control, the state and adjoint state  can be bounded from above by the best approximations of the state and adjoint state variables in finite element space $V_\mathcal{T}$ measured in $L^2$-norm, plus data oscillations, if $\mathcal{T}$ is sufficiently mildly graded. We refer to \cite[Corollary 1]{Demlow} for a similar result for elliptic boundary value problem and \cite{Ciarlet} for the similar C\'ea's lemma in energy norm.

\begin{Theorem}\label{Thm:2}
Let $(u,y,p)\in U_{ad}\times H_0^1(\Omega)\times H_0^1(\Omega)$ be the solution of optimal control problem (\ref{OCP})-(\ref{OCP_state}) and $(u_\mathcal{T},y_\mathcal{T},p_\mathcal{T})\in U_{ad}\times V_\mathcal{T}\times V_\mathcal{T}$ be the solution of the discrete problem (\ref{OCP_state_h})-(\ref{OCP_adjoint_h}). Then we have
\begin{eqnarray}
\|u-u_\mathcal{T}\|_{0,\Omega}+\|y-y_\mathcal{T}\|_{0,\Omega}+\|p-p_\mathcal{T}\|_{0,\Omega}
&\lesssim& \inf\limits_{v_\mathcal{T}\in V_\mathcal{T}}\big(\|y-v_\mathcal{T}\|_{0,\Omega}+{\rm osc}(f+u-\mathcal{L}v_\mathcal{T},\mathcal{T})\big)\nonumber\\
&&+\inf\limits_{w_\mathcal{T}\in V_\mathcal{T}}\big(\|p-w_\mathcal{T}\|_{0,\Omega}+{\rm osc}(y-y_d-\mathcal{L}^*w_\mathcal{T},\mathcal{T})\big).\label{cea}
\end{eqnarray}
\end{Theorem} 
\begin{proof}
From the standard error estimate for elliptic optimal control problem with variational control discretization (see \cite[Sec. 3, Thm. 3.4]{Hinze09book}) we have 
\begin{eqnarray}
&&\|u-u_\mathcal{T}\|_{0,\Omega}+\|y-y_\mathcal{T}\|_{0,\Omega}+\|p-p_\mathcal{T}\|_{0,\Omega}\nonumber\\
&\lesssim&\|y-S_\mathcal{T}u\|_{0,\Omega}+\|p-S_\mathcal{T}^*(y-y_d)\|_{0,\Omega}.\nonumber
\end{eqnarray}
Recall that $\|\nabla h_\mathcal{T}\|_{0,\infty,\Omega}\leq \mu$. For any $v_\mathcal{T},w_\mathcal{T}\in V_\mathcal{T}$, it follows from Lemmas \ref{Lemma:1}, \ref{Lemma:5} and \ref{Lemma:6}  that
\begin{eqnarray}
&&\|u-u_\mathcal{T}\|_{0,\Omega}+\|y-y_\mathcal{T}\|_{0,\Omega}+\|p-p_\mathcal{T}\|_{0,\Omega}\nonumber\\
&\lesssim&\||h_\mathcal{T}(y-S_\mathcal{T}u)\||+\||h_\mathcal{T}(p-S_\mathcal{T}^*(y-y_d))\||\nonumber\\
&\lesssim&\||h_\mathcal{T}(y-v_\mathcal{T})\||+\mu(\|y-S_\mathcal{T}u\|_{0,\Omega}+\|y-v_\mathcal{T}\|_{0,\Omega})\nonumber\\
&&+\||h_\mathcal{T}(p-w_\mathcal{T})\||+\mu(\|p-S_\mathcal{T}^*(y-y_d)\|_{0,\Omega}+\|p-w_\mathcal{T}\|_{0,\Omega})\nonumber\\
&\lesssim&(1+\mu)\|y-v_\mathcal{T}\|_{0,\Omega}+{\rm osc}(f+u-\mathcal{L}v_\mathcal{T},\mathcal{T})+\mu (\|y-y_\mathcal{T}\|_{0,\Omega}+\|y_\mathcal{T}-S_\mathcal{T}u\|_{0,\Omega})\nonumber\\
&&+(1+\mu)\|p-w_\mathcal{T}\|_{0,\Omega}+{\rm osc}(y-y_d-\mathcal{L}^*w_\mathcal{T},\mathcal{T})+\mu (\|p-p_\mathcal{T}\|_{0,\Omega}+\|p_\mathcal{T}-S_\mathcal{T}^*(y-y_d)\|_{0,\Omega})\nonumber\\
&\lesssim&(1+\mu)\|y- v_\mathcal{T}\|_{0,\Omega}+{\rm osc}(f+u-\mathcal{L}v_\mathcal{T},\mathcal{T})+\mu (\|y-y_\mathcal{T}\|_{0,\Omega}+\|u-u_\mathcal{T}\|_{0,\Omega})\nonumber\\
&&+(1+\mu)\|p-w_\mathcal{T}\|_{0,\Omega}+{\rm osc}(y-y_d-\mathcal{L}^*w_\mathcal{T},\mathcal{T})+\mu (\|p-p_\mathcal{T}\|_{0,\Omega}+\|y-y_\mathcal{T}\|_{0,\Omega}),\label{cea_1}
\end{eqnarray}
where we used the discrete stability of elliptic equation in the last estimate. By taking $\mu$ sufficiently small and $v_\mathcal{T},w_\mathcal{T}$ arbitrary we complete the proof.
\end{proof}

\begin{Remark}\label{Rem:1}
Compared to Theorem \ref{Thm:2} we have an alternative result: there exists a constant $C_5$ independent of the mesh size such that
\begin{eqnarray}
&&\|u-u_\mathcal{T}\|_{0,\Omega}+\|y-y_\mathcal{T}\|_{0,\Omega}+\|p-p_\mathcal{T}\|_{0,\Omega}+{\rm osc}_\mathcal{T}\nonumber\\
&\leq& C_5\Big(\inf\limits_{v_\mathcal{T}\in V_\mathcal{T}}\big(\|y-v_\mathcal{T}\|_{0,\Omega}+{\rm osc}(v_\mathcal{T}-y_d,\mathcal{T})+{\rm osc}(\mathcal{L}v_\mathcal{T},\mathcal{T})\big)\nonumber\\
&&+\inf\limits_{w_\mathcal{T}\in V_\mathcal{T}}\big(\|p-w_\mathcal{T}\|_{0,\Omega}+{\rm osc}(f+P_{[a,b]}(w_\mathcal{T}),\mathcal{T})+{\rm osc}(\mathcal{L}^*w_\mathcal{T},\mathcal{T})\big)\Big)\label{cea_2}
\end{eqnarray}
provided that $h_0\ll 1$. In fact, we can conclude from the triangle inequality that
\begin{eqnarray}
{\rm osc}(f+u,\mathcal{T})+{\rm osc}(y-y_d,\mathcal{T})&\lesssim& {\rm osc}(f+P_{[a,b]}(w_\mathcal{T}),\mathcal{T})+{\rm osc}(v_\mathcal{T}-y_d,\mathcal{T})\nonumber\\
&&+\|h_\mathcal{T}^2(u-P_{[a,b]}(w_\mathcal{T}))\|_{0,\Omega}+\|h_\mathcal{T}^2(y-v_\mathcal{T})\|_{0,\Omega}.\nonumber
\end{eqnarray}
From the Lipschitz property of the projection operator $P_{[a,b]}$ we have 
\begin{eqnarray}
\|u-P_{[a,b]}(w_\mathcal{T})\|_{0,\Omega}\lesssim \|p-w_\mathcal{T}\|_{0,\Omega}.\nonumber
\end{eqnarray}
Now it remains to estimate ${\rm osc}_\mathcal{T}$. It follows from the  inverse inequality that
\begin{eqnarray}
{\rm osc}(\mathcal{L}(y_\mathcal{T}-v_\mathcal{T}),T)+{\rm osc}(\mathcal{L}^*(p_\mathcal{T}-w_\mathcal{T}),T)\lesssim h_T(\|y_\mathcal{T}-v_\mathcal{T}\|_{0,T}+\|p_\mathcal{T}-w_\mathcal{T}\|_{0,T}).\nonumber
\end{eqnarray}
Therefore,
\begin{eqnarray}
&&{\rm osc}(f+u_\mathcal{T}-\mathcal{L}y_\mathcal{T},\mathcal{T})+{\rm osc}(y_\mathcal{T}-y_d-\mathcal{L}^*p_\mathcal{T},\mathcal{T})\nonumber\\
&\lesssim&{\rm osc}(f+P_{[a,b]}(w_\mathcal{T}),\mathcal{T})+{\rm osc}(v_\mathcal{T}-y_d,\mathcal{T})+\|h_\mathcal{T}^2(u_\mathcal{T}-P_{[a,b]}(w_\mathcal{T}))\|_{0,\Omega}+\|h_\mathcal{T}^2(y_\mathcal{T}-v_\mathcal{T})\|_{0,\Omega}\nonumber\\
&&+{\rm osc}(\mathcal{L}v_\mathcal{T},\mathcal{T})+{\rm osc}(\mathcal{L}^*w_\mathcal{T},\mathcal{T})+{\rm osc}(\mathcal{L}(y_\mathcal{T}-v_\mathcal{T}),\mathcal{T})+{\rm osc}(\mathcal{L}^*(p_\mathcal{T}-w_\mathcal{T}),\mathcal{T})\nonumber\\
&\lesssim&{\rm osc}(f+P_{[a,b]}(w_\mathcal{T}),\mathcal{T})+{\rm osc}(v_\mathcal{T}-y_d,\mathcal{T})+{\rm osc}(\mathcal{L}v_\mathcal{T},\mathcal{T})+{\rm osc}(\mathcal{L}^*w_\mathcal{T},\mathcal{T})\nonumber\\
&&+\|p-w_\mathcal{T}\|_{0,\Omega}+\|y-v_\mathcal{T}\|_{0,\Omega}+\|h_\mathcal{T}(p-p_\mathcal{T})\|_{0,\Omega}+\|h_\mathcal{T}(y-y_\mathcal{T})\|_{0,\Omega}.\nonumber
\end{eqnarray}
Combining the above results, the fact that $h_\mathcal{T}\leq h_0$ and using (\ref{cea}) we can conclude (\ref{cea_2}) provided that $h_0\ll 1$ and $\mu$ sufficiently small.
\end{Remark}

\subsection{Convergence analysis of AFEM for OCPs in $L^2$}
In this subsection we will prove the convergence of $L^2$-norm based AFEM for solving optimal control problems. In the following we assume that $(u,y,p)\in  U_{ad}\times H_0^1(\Omega)\times H_0^1(\Omega)$ is the solution of optimal control problem (\ref{OCP})-(\ref{OCP_state}) and $(u_k,y_k,p_k) \in U_{ad}\times  V_k\times V_k$ is the solution of the discrete problem (\ref{OCP_state_h})-(\ref{OCP_adjoint_h}) generated by the adaptive Algorithm \ref{Alg:3.1}. 

At first we prove some quasi-orthogonality properties.
\begin{Lemma}\label{Lemma:9}
For any $\epsilon>0$ there hold
\begin{eqnarray}
&&\||h_{k+1}(y-y_{k+1})\||^2+\||h_{k+1}(y_{k+1}-y_{k})\||^2-(1+\epsilon)\||h_{k}(y-y_{k})\||^2\nonumber\\
&\lesssim& \epsilon^{-1}(\mu^2+h_0^2)(\|u-u_{k+1}\|_{0,\Omega}^2+\|y-y_{k+1}\|_{0,\Omega}^2+\| y-y_{k}\|_{0,\Omega}^2),\label{quasi_y}\\
&&\||h_{k+1}(p-p_{k+1})\||^2+\||h_{k+1}(p_{k+1}-p_{k})\||^2-(1+\epsilon)\||h_{k}(p-p_{k})\||^2\nonumber\\
&\lesssim& \epsilon^{-1}(\mu^2+h_0^2)(\|y-y_{k+1}\|_{0,\Omega}^2+\|p-p_{k+1}\|_{0,\Omega}^2+\| p-p_{k}\|_{0,\Omega}^2).\label{quasi_p}
\end{eqnarray}
\end{Lemma} 
\begin{proof}

At first we estimate (\ref{quasi_y}). We use the abbreviations $e_k=y-y_{k}$ and $\tilde e_k=y_{k+1}-y_{k}$, then we have
\begin{eqnarray}
\||h_{k+1}e_{k+1}\||^2=\||h_{k+1}e_{k}\||^2-\||h_{k+1}\tilde e_{k}\||^2-2a( h_{k+1}e_{k+1},h_{k+1}\tilde e_{k}).\label{y_split}
\end{eqnarray}
An elementary calculation gives
\begin{eqnarray}
&&|(A\nabla (h_{k+1}e_{k+1}),\nabla (h_{k+1}\tilde e_{k}))+(a_0h_{k+1}e_{k+1},h_{k+1}\tilde e_{k})|\nonumber\\
&=&|(A\nabla e_{k+1},\nabla (h_{k+1}^2\tilde e_{k}))+(A|\nabla h_{k+1}|^2e_{k+1},\tilde e_{k})\nonumber\\
&&-(Ae_k\nabla h_{k+1},\nabla (h_{k+1}e_{k+1}))+(Ae_{k+1}\nabla h_{k+1},\nabla(h_{k+1}e_k))+(a_0h_{k+1}e_{k+1},h_{k+1}\tilde e_{k})|\nonumber\\
&\lesssim&\|e_k\nabla h_{k+1}\|_{0,\Omega}\||h_{k+1}e_{k+1}\||+\|e_{k+1}\nabla h_{k+1}\|_{0,\Omega}\||h_{k+1}e_k\||\nonumber\\
&&+\|e_{k+1}\nabla h_{k+1}\|_{0,\Omega}\|\tilde e_k\nabla h_{k+1}\|_{0,\Omega}+|(A\nabla e_{k+1},\nabla (h_{k+1}^2\tilde e_{k}))+(a_0e_{k+1},h_{k+1}^2\tilde e_{k})|.\label{orth_y_1}
\end{eqnarray}
It remains to estimate $|(A\nabla e_{k+1},\nabla (h_{k+1}^2\tilde e_{k}))+(a_0e_{k+1},h_{k+1}^2\tilde e_{k})|$. We divide the estimate into two steps. Firstly, from the orthogonality property we have
\begin{eqnarray}
&&|(A\nabla (Su_{k+1}-y_{k+1}),\nabla (h_{k+1}^2\tilde e_{k}))+(a_0(Su_{k+1}-y_{k+1}),h_{k+1}^2\tilde e_{k})|\nonumber\\
&=&|(A\nabla (Su_{k+1}-y_{k+1}),\nabla (h_{k+1}^2\tilde e_{k}-\Pi_{k+1}(h_{k+1}^2\tilde e_{k})))\nonumber\\
&&+(a_0(Su_{k+1}-y_{k+1}),h_{k+1}^2\tilde e_{k}-\Pi_{k+1}(h_{k+1}^2\tilde e_{k}))|\nonumber\\
&\leq& \sum\limits_{T\in \mathcal{T}_{k+1}}\||Su_{k+1}-y_{k+1}\||_{T}\||h_{k+1}^2\tilde e_{k}-\Pi_{k+1}(h_{k+1}^2\tilde e_{k}))\||_{T}.\label{orth_y_2}
\end{eqnarray}
For each $T\in \mathcal{T}_{k+1}$ we know that $h_{k+1}|_{T}$ and $\tilde e_{k}$ are linear. So $|\nabla^2h_{k+1}^2|\lesssim |\nabla h_{k+1}|^2$, $\nabla^2 h_{k+1}=0$ and $\nabla^2\tilde e_k=0$. By using the inverse inequality, the standard interpolation error estimate (\cite{Ciarlet}), the fact that $\|\nabla h_{k+1}\|_{0,\infty,\Omega}\lesssim 1$ and (\ref{shape_regular}) we can derive
\begin{eqnarray}
&&\||h_{k+1}^2\tilde e_{k}-\Pi_{k+1}(h_{k+1}^2\tilde e_{k}))\||_{T}\nonumber\\
&\lesssim& h_T\|\nabla^2(h_{k+1}^2\tilde e_{k})\|_{0,T}\nonumber\\
&\lesssim& h_T(\|\nabla h_{k+1}^2\nabla \tilde e_{k})\|_{0,T}+\|\nabla^2 h_{k+1}^2\tilde e_{k})\|_{0,T})\nonumber\\
&\lesssim& h_T\|\nabla h_{k+1}\|_{0,\infty,T}(\|h_{k+1}\|_{0,\infty,T}\|\nabla \tilde e_{k}\|_{0,T}+\|\nabla h_{k+1}\|_{0,\infty,T}\|\tilde e_{k}\|_{0,T})\nonumber\\
&\lesssim&\|\nabla h_{k+1}\|_{0,\infty,T}\|\tilde e_{k}\|_{0,T}(\|h_{k+1}\|_{0,\infty,T}+h_T\|\nabla h_{k+1}\|_{0,\infty,T})\nonumber\\
&\lesssim&\|\nabla h_{k+1}\|_{0,\infty,T}\|h_{k+1}\|_{0,\infty,T}\|\tilde e_{k}\|_{0,T}.\label{orth_y_3}
\end{eqnarray}
From (\ref{orth_y_2}),  (\ref{orth_y_3}) and the stability of elliptic equation we conclude that
\begin{eqnarray}
&&|(A\nabla (Su_{k+1}-y_{k+1}),\nabla (h_{k+1}^2\tilde e_{k}))+(a_0(Su_{k+1}-y_{k+1}),h_{k+1}^2\tilde e_{k})|\nonumber\\
&\lesssim& \sum\limits_{T\in \mathcal{T}_{k+1}}\||Su_{k+1}-y_{k+1})\||_{T}\|\nabla h_{k+1}\|_{0,\infty,T}\|h_{k+1}\|_{0,\infty,T}\|\tilde e_{k}\|_{0,T}\nonumber\\
&\lesssim& (\|h_{k+1}(\nabla e_{k+1}+e_{k+1})\|_{0,\Omega}+\|h_{k+1}(\nabla (y-Su_{k+1})+(y-Su_{k+1}))\|_{0,\Omega})\nonumber\\
&&\|\nabla h_{k+1}\|_{0,\infty,\Omega}\|\tilde e_{k}\|_{0,\Omega}\nonumber\\
&\lesssim&( \||h_{k+1}e_{k+1}\||+\|\nabla h_{k+1}e_{k+1}\|_{0,\Omega}+\|h_{k+1}\|_{0,\infty,\Omega}\|u-u_{k+1}\|_{0,\Omega})\|\nabla h_{k+1}\|_{0,\infty,\Omega}\|\tilde e_{k}\|_{0,\Omega}.\label{orth_y_6}
\end{eqnarray}
Secondly, from (\ref{OCP_state_weak}) we deduce
\begin{eqnarray}
&&|(A\nabla (y-Su_{k+1}),\nabla (h_{k+1}^2\tilde e_{k}))+(a_0 (y-Su_{k+1}),h_{k+1}^2\tilde e_{k})|\nonumber\\
&=& |(u-u_{k+1},h_{k+1}^2\tilde e_{k})|\nonumber\\
&\leq&\|h_{k+1}\|_{0,\infty,\Omega}^2\|u-u_{k+1}\|_{0,\Omega}\|\tilde e_{k}\|_{0,\Omega}.\label{orth_y_7}
\end{eqnarray}
Inserting the above estimates into (\ref{orth_y_1}) and using Young's inequality we are led to 
\begin{eqnarray}
&&|(A\nabla (h_{k+1}e_{k+1}),\nabla (h_{k+1}\tilde e_{k}))+(a_0h_{k+1}e_{k+1},h_{k+1}\tilde e_{k})|\nonumber\\
&\leq& {\delta \over 2}(\||h_{k+1}e_{k+1}\||^2+\||h_{k+1} e_{k}\||^2)+\|h_{k+1}\|_{0,\infty,\Omega}^2\|u-u_{k+1}\|_{0,\Omega}^2\nonumber\\
&&+C(1+{1\over \delta})(\|\nabla h_{k+1}\|_{0,\infty,\Omega}^2+\|h_{k+1}\|_{0,\infty,\Omega}^2)(\|e_{k+1}\|_{0,\Omega}^2+\| e_{k}\|_{0,\Omega}^2).\label{orth_y_4}
\end{eqnarray}
Combining (\ref{y_split}) and (\ref{orth_y_4}) yields
\begin{eqnarray}
(1-\delta)\||h_{k+1}e_{k+1}\||^2&\leq& (1+\delta)\||h_{k+1}e_{k}\||^2-\||h_{k+1}\tilde e_{k}\||^2+2\|h_{k+1}\|_{0,\infty,\Omega}^2\|u-u_{k+1}\|_{0,\Omega}^2\nonumber\\
&&+2C(1+{1\over \delta})(\|\nabla h_{k+1}\|_{0,\infty,\Omega}^2+\|h_{k+1}\|_{0,\infty,\Omega}^2)(\|e_{k+1}\|_{0,\Omega}^2+\| e_{k}\|_{0,\Omega}^2).
\end{eqnarray}
Dividing both sides by $1-\delta$ and choosing $\frac{1+\delta}{1-\delta}$ as $1+\epsilon$ we arrive at
\begin{eqnarray}
\||h_{k+1}e_{k+1}\||^2&\lesssim&(1+\epsilon)\||h_{k+1}e_{k}\||^2-\||h_{k+1}\tilde e_{k}\||^2+\|h_{k+1}\|_{0,\infty,\Omega}^2\|u-u_{k+1}\|_{0,\Omega}^2\nonumber\\
&&+\epsilon^{-1}(\|\nabla h_{k+1}\|_{0,\infty,\Omega}^2+\|h_{k+1}\|_{0,\infty,\Omega}^2)(\|e_{k+1}\|_{0,\Omega}^2+\| e_{k}\|_{0,\Omega}^2).\label{orth_y_5}
\end{eqnarray}
Note that $h_{k+1}\leq h_k$ and $\|\nabla h_{k+1}\|_{0,\infty,\Omega}\leq \mu$ we have
\begin{eqnarray}
\||h_{k+1}e_{k}\||&\lesssim& \|h_{k+1}\nabla e_k\|_{0,\Omega}+\|\nabla h_{k+1}\|_{0,\infty,\Omega}\|e_k\|_{0,\Omega}+ \|h_{k+1} e_k\|_{0,\Omega}\nonumber\\
&\lesssim &\|h_{k}\nabla e_k\|_{0,\Omega}+ \|h_{k} e_k\|_{0,\Omega}+\mu\|e_k\|_{0,\Omega}\nonumber\\
&\lesssim &\||h_{k}e_k\||+2\mu\|e_k\|_{0,\Omega},
\end{eqnarray}
so for any $\epsilon >0$ we have
\begin{eqnarray}
\||h_{k+1}e_{k}\||^2\lesssim (1+\epsilon)\||h_{k}e_k\||^2+(1+{1\over\epsilon})4\mu^2\|e_k\|_{0,\Omega}^2,
\end{eqnarray}
substituting the above result into (\ref{orth_y_5}) and using $\|h_{k+1}\|_{0,\infty,\Omega}\leq h_0$ we complete the proof of (\ref{quasi_y}) with $(1+\epsilon)$ replaced by $(1+\epsilon)^2$ which are equivalent. The proof of (\ref{quasi_p}) is very similar and we omit it here.
\end{proof}

We also need the following estimator reduction property, the proof is very similar to \cite[Corollary]{Cascon}.
\begin{Lemma}\label{Lemma:10}
For any $\delta \in (0,1]$ there hold
\begin{eqnarray}
&&\eta_{k+1,y}^2(u_{k+1},y_{k+1},\mathcal{T}_{k+1})-(1+\delta)\Big(\eta_{k,y}^2(u_{k},y_{k},\mathcal{T}_{k})-\lambda \eta_{k,y}^2(u_{k},y_{k},\mathcal{M}_{k})\Big)\nonumber\\
&\leq&\delta^{-1}\Big(\||h_{k+1}(y_{k+1}-y_{k})\||^2+\mu^2\|y_{k+1}-y_{k}\|_{0,\Omega}^2+h_0^4\|u_{k+1}-u_k\|^2_{0,\Omega}\Big),\label{err_red_y}\\
&&\eta_{i+1,p}^2(y_{k+1},p_{k+1},\mathcal{T}_{k+1})-(1+\delta)\Big(\eta_{k,p}^2(y_{k},p_{k},\mathcal{T}_{k})-\lambda\eta_{k,p}^2(y_{k},p_{k},\mathcal{M}_{k})\Big)\nonumber\\
&\leq&\delta^{-1}\Big(\||h_{k+1}(p_{k+1}-p_{k})\||^2+\mu^2\|p_{k+1}-p_{k}\|_{0,\Omega}^2+h_0^4\|y_{k+1}-y_k\|^2_{0,\Omega}\Big),\label{err_red_p}
\end{eqnarray}
where $\lambda = 1-2^{-{{3r}\over d}}$ and $r$ is the number of bisections for marked elements in $\mathcal{M}_{k}$ during the local mesh refinement. 
\end{Lemma} 
\begin{proof}
It follows from (\ref{estimatory_pert}) and Young's inequality that for any $\delta \in (0,1]$ and $T\in \mathcal{T}_{k+1}$ there holds
\begin{eqnarray}
&&\eta_{k+1,y}^2(u_{k+1},y_{k+1},T)-(1+\delta)\eta_{k+1,y}^2(u_{k},y_{k},T)\nonumber\\
&\lesssim& \delta^{-1}\Big(\||h_{k+1}(y_{k+1}-y_{k})\||_{\tilde{\omega}_T}^2+\|\nabla h_{k+1}\|_{0,\infty,\tilde{\omega}_T}^2\|y_{k+1}-y_{k}\|_{0,\tilde{\omega}_T}^2+h_{T}^4\|u_{k+1}-u_k\|^2_{0,\tilde{\omega}_T}\Big).\nonumber
\end{eqnarray}
Summing over $T\in \mathcal{T}_{k+1}$ and using the fact that the triangulation is shape regular, (\ref{mesh_grading}) and $h_{k+1}\leq h_0$, we have
\begin{eqnarray}
&&\eta_{k+1,y}^2(u_{k+1},y_{k+1},\mathcal{T}_{k+1})-(1+\delta)\eta_{k+1,y}^2(u_{k},y_{k},\mathcal{T}_{k+1})\nonumber\\
&\lesssim& \delta^{-1}\Big(\||h_{k+1}(y_{k+1}-y_{k})\||^2+\mu^2\|y_{k+1}-y_{k}\|_{0,\Omega}^2+h_0^4\|u_{k+1}-u_k\|^2_{0,\Omega}\Big).\label{est_red_y_1}
\end{eqnarray}
For $T'\in \mathcal{T}_k$ we define $\mathcal{T}_{T'}=\{T\in \mathcal{T}_{k+1}: T\subset T'\}$. From the definition of bisection algorithm we know that for a marked element $T'\in\mathcal{M}_{k}$ and $T\in \mathcal{T}_{T'}$ there holds $h_T\leq 2^{-{r\over d}}h_{T'}$ and $[A\nabla y_k]=0$ across the edges of $T$ which lie in the interior of $T'$. Therefore,
\begin{eqnarray}
\eta_{k+1,y}^2(u_{k},y_{k},\mathcal{T}_{T'})\leq 2^{-{{3r}\over d}}\eta_{k,y}^2(u_{k},y_{k},T').
\end{eqnarray}
For $T'\in \mathcal{T}_k\backslash \mathcal{M}_k$, we have the monotonicity property $\eta_{k+1,y}(u_{k},y_{k},\mathcal{T}_{T'})\leq \eta_{k,y}(u_{k},y_{k},T')$ (see, for instance, \cite[Remark 2.1]{Cascon}). Summing over $T\in \mathcal{T}_{k+1}$ we obtain
\begin{eqnarray}
\eta_{k+1,y}^2(u_{k},y_{k},\mathcal{T}_{k+1})&\leq& 2^{-{{3r}\over d}}\eta_{k,y}^2(u_{k},y_{k},\mathcal{M}_k)+\eta_{k,y}^2(u_{k},y_{k},\mathcal{T}_k\backslash\mathcal{M}_k)\nonumber\\
&=&\eta_{k,y}^2(u_{k},y_{k},\mathcal{T}_k)-\lambda \eta_{k,y}^2(u_{k},y_{k},\mathcal{M}_k).\label{est_red_y_2}
\end{eqnarray}
Combining (\ref{est_red_y_1}) and (\ref{est_red_y_2}) prove (\ref{err_red_y}). Similarly, we can prove (\ref{err_red_p}).
\end{proof}

Now we are in the position to prove the contraction property for the  weighted energy norm errors of the state and adjoint state.
\begin{Theorem}\label{Thm:3}
Let $(u,y,p)\in U_{ad}\times H_0^1(\Omega)\times H_0^1(\Omega)$ be the solution of optimal control problem (\ref{OCP})-(\ref{OCP_state}) and $(u_k,y_k,p_k)\in U_{ad}\times V_k\times V_k$ be the solution of the discrete problem (\ref{OCP_state_h})-(\ref{OCP_adjoint_h}) generated by the adaptive Algorithm \ref{Alg:3.1}. Then there exist constant $\gamma>0$ and $\nu\in (0,1)$ depending on $c_\mathbb{T}$, $C_\mathbb{T}$, $C_{\rm reg}$, the shape regularity of $\mathcal{T}_0$, the parameter $\theta$ in Algorithm \ref{Alg:3.0} and the number of times $r$ that each element in $\mathcal{T}_k$ is bisected, such that for sufficiently small $\mu$ it holds
\begin{eqnarray}
&&\||h_{k+1}(y-y_{k+1})\||^2+\||h_{k+1}(p-p_{k+1})\||^2+\gamma \eta_{k+1}^2(\mathcal{T}_{k+1})\nonumber\\
&\leq& \nu^2\Big(\||h_{k}(y-y_{k})\||^2+\||h_{k}(p-p_{k})\||^2+\gamma \eta_{k}^2(\mathcal{T}_k)\Big)
\end{eqnarray}
provided that $h_0\ll 1$.
\end{Theorem}
\begin{proof}
We use the abbreviations 
\begin{eqnarray}
E_{k}^2:=\||h_{k}(y-y_{k})\||^2+\||h_{k}(p-p_{k})\||^2,\nonumber\\
\tilde{E}_{k}^2:=\||h_{k+1}(y_{k+1}-y_k)\||^2+\||h_{k+1}(p_{k+1}-p_k)\||^2.\nonumber
\end{eqnarray}
From Lemma \ref{Lemma:12}  we have
\begin{eqnarray}
E_{k}^2&\leq& C_3 \eta_k^2(\mathcal{T}_k).\label{step_1}
\end{eqnarray}
We can conclude from Lemma \ref{Lemma:9} and Theorem \ref{Thm:1} that
\begin{eqnarray}
E_{k+1}^2&\leq& (1+\epsilon)E_k^2-\tilde{E}_k^2+C_6\epsilon^{-1}(\mu^2+h_0^2)\big(\eta_k^2(\mathcal{T}_k)+\eta_{k+1}^2(\mathcal{T}_{k+1})\big).\label{step_2}
\end{eqnarray}
Moreover, it follows from Lemma \ref{Lemma:10}, the triangle inequality and Theorem \ref{Thm:1} that
\begin{eqnarray}
\eta_{k+1}^2(\mathcal{T}_{k+1})&\leq& (1+\delta )(\eta_{k}^2(\mathcal{T}_k)-\lambda\eta_{k}^2(\mathcal{M}_k))+C_7\delta^{-1}\Big(\tilde{E}_k^2+(\mu^2+h_0^4)\big(\eta_k^2(\mathcal{T}_k)+\eta_{k+1}^2(\mathcal{T}_{k+1})\big)\Big)\nonumber\\
&\leq&(1+\delta )(1-\lambda\theta^2)\eta_{k}^2(\mathcal{T}_k)+C_7\delta^{-1}\Big(\tilde{E}_k^2+(\mu^2+h_0^4)\big(\eta_k^2(\mathcal{T}_k)+\eta_{k+1}^2(\mathcal{T}_{k+1})\big)\Big)\nonumber\\
&\leq&(1+\delta )\Big((1-{1\over 2}\lambda\theta^2)\eta_{k}^2(\mathcal{T}_k)-\frac{1}{2C_3}\lambda\theta^2E_k^2\Big)\nonumber\\
&&+C_7\delta^{-1}\Big(\tilde{E}_k^2+(\mu^2+h_0^4)\big(\eta_k^2(\mathcal{T}_k)+\eta_{k+1}^2(\mathcal{T}_{k+1})\big)\Big),\label{step_3}
\end{eqnarray}
where we used D\"orfler's marking strategy in Algorithm \ref{Alg:3.0} in the second inequality and  (\ref{step_1}) in the third inequality. 

Now we multiply (\ref{step_3}) with $\tilde \gamma =\delta C_7^{-1}$, the sum of which with (\ref{step_2}) gives
\begin{eqnarray}
E_{k+1}^2+\tilde \gamma\eta_{k+1}^2(\mathcal{T}_{k+1})&\leq& (1+\delta )(1-{1\over 2}\lambda\theta^2)\tilde \gamma\eta_{k}^2(\mathcal{T}_k)+\Big((1+\epsilon)- (1+\delta )\frac{\tilde\gamma}{2C_3}\lambda\theta^2\Big)E_k^2\nonumber\\
&&+(\mu^2+h_0^2)(1+C_6\epsilon^{-1})\big(\eta_k^2(\mathcal{T}_k)+\eta_{k+1}^2(\mathcal{T}_{k+1})\big).\label{step_4}
\end{eqnarray}
We set $\delta$ and $\epsilon$ sufficiently small such that there holds
\begin{eqnarray}
\tilde\nu^2:=\max\Big\{(1+\delta )(1-{1\over 2}\lambda\theta^2),(1+\epsilon)- (1+\delta )\frac{\tilde\gamma}{2C_3}\lambda\theta^2\Big\}<1.\label{tilde_nu}
\end{eqnarray}
It follows from (\ref{step_4}) that
\begin{eqnarray}
&&E_{k+1}^2+\Big(\tilde\gamma-(\mu^2+h_0^2)(1+C_6\epsilon^{-1})\Big)\eta_{k+1}^2(\mathcal{T}_{k+1})\nonumber\\
&\leq& \tilde\nu^2E_k^2+\Big(\tilde\nu^2\tilde\gamma+(\mu^2+h_0^2)(1+C_6\epsilon^{-1})\Big)\eta_k^2(\mathcal{T}_k).
\end{eqnarray}
Choosing $\mu$ sufficiently small and $h_0\ll 1$ such that 
\begin{eqnarray}
0<\frac{\tilde\nu^2\tilde\gamma+(\mu^2+h_0^2)(1+C_6\epsilon^{-1})}{\tilde\gamma-(\mu^2+h_0^2)(1+C_6\epsilon^{-1})}\leq \frac{1+\tilde \nu^2}{2}:=\nu^2.
\end{eqnarray}
By choosing $\gamma:=\tilde \gamma-(\mu^2+h_0^2)(1+C_6\epsilon^{-1})$ we obtain
\begin{eqnarray}
E_{k+1}^2+\gamma\eta_{k+1}^2(\mathcal{T}_{k+1})
&\leq& \tilde\nu^2E_k^2+\gamma\frac{\tilde\nu^2\tilde\gamma+(\mu^2+h_0^2)(1+C_6\epsilon^{-1})}{\tilde\gamma-(\mu^2+h_0^2)(1+C_6\epsilon^{-1})}\eta_k^2(\mathcal{T}_k)\nonumber\\
&\leq&\tilde\nu^2E_k^2+\gamma\nu^2\eta_k^2(\mathcal{T}_k)\nonumber\\
&\leq&\nu^2\Big(\frac{2\tilde\nu^2}{1+\tilde \nu^2}E_k^2+\gamma\eta_k^2(\mathcal{T}_k)\Big)\nonumber\\
&\leq&\nu^2\Big(E_k^2+\gamma\eta_k^2(\mathcal{T}_k)\Big)
\end{eqnarray}
in view of (\ref{tilde_nu}). This completes the proof.
\end{proof}

The convergence of the $L^2$-norm errors for both the control, the state and adjoint state is the direct consequence of that of the weighted energy norm errors and the equivalence between them.
\begin{Theorem}\label{Thm:4}
Let $(u,y,p)\in U_{ad}\times H_0^1(\Omega)\times H_0^1(\Omega)$ be the solution of optimal control problem (\ref{OCP})-(\ref{OCP_state}) and $(u_k,y_k,p_k)\in U_{ad}\times V_k\times V_k$ be the solution of the discrete problem (\ref{OCP_state_h})-(\ref{OCP_adjoint_h}) generated by the adaptive Algorithm \ref{Alg:3.1}. Let the assumptions in Theorem \ref{Thm:3} be fulfilled, it holds that for $k\geq l$
\begin{eqnarray}
&&\|u-u_{k+1}\|_{0,\Omega}+\|y-y_{k+1}\|_{0,\Omega}+\|p-p_{k+1}\|_{0,\Omega}+{\rm osc}_{k+1}\nonumber\\
&\lesssim& \nu^{k-l}\Big(\|u-u_{l}\|_{0,\Omega}+\|y-y_{l}\|_{0,\Omega}+\|p-p_{l}\|_{0,\Omega}+{\rm osc}_{l}\Big)
\end{eqnarray}
provided that $h_0\ll 1$.
\end{Theorem}
\begin{proof}
From the dominance of the indicator over oscillation (see \cite[Remark 2.1]{Cascon}) we have
\begin{eqnarray}
{\rm osc}_{k+1}(f+u_{k+1}-\mathcal{L}y_{k+1},\mathcal{T}_{k+1})\leq \eta_{k+1,y}(u_{k+1},y_{k+1},\mathcal{T}_{k+1}),\nonumber\\
{\rm osc}_{k+1}(y_{k+1}-y_d-\mathcal{L}^*p_{k+1},\mathcal{T}_{k+1})\leq \eta_{k+1,p}(y_{k+1},p_{k+1},\mathcal{T}_{k+1}),\nonumber
\end{eqnarray}
which in turn implies
\begin{eqnarray}
{\rm osc}_{k+1}^2\leq \eta_{k+1}^2(\mathcal{T}_{k+1}),
\end{eqnarray}
this together with Lemma \ref{Lemma:11} yields
\begin{eqnarray}
&&(\|u-u_{k+1}\|_{0,\Omega}+\|y-y_{k+1}\|_{0,\Omega}+\|p-p_{k+1}\|_{0,\Omega}+{\rm osc}_{k+1})^2\nonumber\\
&\lesssim& 
\||h_{k+1}(y-y_{k+1})\||^2+\||h_{k+1}(p-p_{k+1})\||^2+\gamma\eta_{k+1}^2(\mathcal{T}_{k+1}).\label{equi_1}
\end{eqnarray}
On the other hand, it follows from Lemma \ref{Lemma:12} and Theorem \ref{Thm:1} that
\begin{eqnarray}
&&\||h_{k+1}(y-y_{k+1})\||^2+\||h_{k+1}(p-p_{k+1})\||^2+\gamma\eta_{k+1}^2(\mathcal{T}_{k+1})\nonumber\\
&\lesssim& \eta^2_{{k+1}}(\mathcal{T}_{k+1})\nonumber\\
&\lesssim&(\|u-u_{k+1}\|_{0,\Omega}+\|y-y_{k+1}\|_{0,\Omega}+\|p-p_{k+1}\|_{0,\Omega}+{\rm osc}_{k+1})^2.\label{equi_2}
\end{eqnarray}
Combining (\ref{equi_1}) and (\ref{equi_2}) we arrive at
\begin{eqnarray}
&&(\|u-u_{k+1}\|_{0,\Omega}+\|y-y_{k+1}\|_{0,\Omega}+\|p-p_{k+1}\|_{0,\Omega}+{\rm osc}_{k+1})^2\nonumber\\
&\simeq& 
\||h_{k+1}(y-y_{k+1})\||^2+\||h_{k+1}(p-p_{k+1})\||^2+\gamma\eta_{k+1}^2(\mathcal{T}_{k+1}),
\end{eqnarray}
this together with Theorem \ref{Thm:3} completes the proof.
\end{proof}

\section{Complexity of AFEM for the optimal control problem under $L^2$-norm}
\setcounter{equation}{0}
In this section we prove the quasi-optimal complexity of $L^2$-norm based AFEM for solving optimal control problems. To begin with, we follow the idea of \cite{Demlow} (see \cite{Cascon,Stevenson} for the definitions of approximation classes with respect to the energy norm based AFEM) to introduce the approximation class $\mathcal{A}^s$ for $s>0$:
\begin{eqnarray}
\mathcal{A}^s=\Big\{ (y,p)\in H_0^1(\Omega)\times H_0^1(\Omega):\quad \mathcal{L} y,\mathcal{L}^* p\in L^2(\Omega), \ |(y,p)|_{\mathcal{A}^s}<\infty\Big\},\nonumber
\end{eqnarray}
where 
\begin{eqnarray}
|(y,p)|_{\mathcal{A}^s}&:=&\sup\limits_{\varepsilon >0}\varepsilon\inf\limits_{\mathcal{T}\in\mathbb{T}:\inf\limits_{v_\mathcal{T}\in V_\mathcal{T}}(\|y-v_\mathcal{T}\|_{0,\Omega}+{\rm osc}(v_\mathcal{T}-y_d,\mathcal{T})+{\rm osc}(\mathcal{L}v_\mathcal{T},\mathcal{T}))\hfill \atop +\inf\limits_{w_\mathcal{T}\in V_\mathcal{T}}(\|p-w_\mathcal{T}\|_{0,\Omega}+{\rm osc}(f+P_{[a,b]}(w_\mathcal{T}),\mathcal{T})+{\rm osc}(\mathcal{L}^*w_\mathcal{T},\mathcal{T}))\leq\varepsilon\hfill}\Big(\#\mathcal{T}-\#\mathcal{T}_0\Big)^s.\nonumber
\end{eqnarray}
In current paper we assume that $\Omega$ is convex and $V_\mathcal{T}$ is linear. Therefore, it holds $(H^2(\Omega)\cap H_0^1(\Omega))^2\subset \mathcal{A}^s$ with $s=1$ if $d=2$ and $s={2\over 3}$ if $d=3$. However, the class $\mathcal{A}^s$ is much larger than $(H^2(\Omega)\cap H_0^1(\Omega))^2$ which makes the $L^2$-norm based AFEM attractive, although the rate $s$ is already realized with uniform refinements if $y,p\in H^2(\Omega)\cap H_0^1(\Omega)$. We refer to \cite[Sec. 7]{Demlow} for more details.

To prove the optimality of AFEM we shall give the complexity of refinement, we refer to \cite[Lemma 2.3]{Cascon} and \cite{Rob} for related results. The following lemma shows that the difference between the cardinalities of the output and initial partitions can be bounded from above by some multiple of the total number of marked elements.  
\begin{Lemma}\cite[Theorem 4]{Demlow}\label{Lemma:15}
Let $K$ be the total number of calls of  bisection algorithm in \cite[P. 210]{Demlow}, so that $K$ is no larger than the sum of the cardinalities of all sets of marked simplicies. Then for the output partition $\mathcal{T}$, it holds that $\#\mathcal{T}-\#\mathcal{T}_0\lesssim K$.
\end{Lemma}

Then we present the following localized upper bounds for the approximations of the optimal control problems.
\begin{Lemma}\label{Lemma:15}
Let $(u,y,p)\in U_{ad}\times H_0^1(\Omega)\times H_0^1(\Omega)$ be the solution of optimal control problem (\ref{OCP})-(\ref{OCP_state}). Given sufficiently small $\mu$, let $\mathcal{T}\in \mathbb{T}_\mu$ and $\mathcal{T}\subset\tilde{\mathcal{T}}\in \mathbb{T}$, $(u_\mathcal{T},y_\mathcal{T},p_\mathcal{T})\in U_{ad}\times V_\mathcal{T}\times V_\mathcal{T}$ and $(u_{\tilde{\mathcal{T}}},y_{\tilde{\mathcal{T}}},p_{\tilde{\mathcal{T}}})\in U_{ad}\times V_{\tilde{\mathcal{T}}}\times V_{\tilde{\mathcal{T}}}$ be the solutions of the discrete problem (\ref{OCP_state_h})-(\ref{OCP_adjoint_h}) on $\mathcal{T}$ and ${\tilde{\mathcal{T}}}$, respectively. Then  there holds
\begin{eqnarray}
\|u_\mathcal{T}-u_{\tilde{\mathcal{T}}}\|_{0,\Omega}+\|y_\mathcal{T}-y_{\tilde{\mathcal{T}}}\|_{0,\Omega}+\|p_\mathcal{T}-p_{\tilde{\mathcal{T}}}\|_{0,\Omega}\leq C_4 \eta_{\mathcal{T}} (\mathcal{R}_{\mathcal{T}\rightarrow \tilde{\mathcal{T}}}),\label{localized_OCP}
\end{eqnarray}
where $\mathcal{R}_{\mathcal{T}\rightarrow \tilde{\mathcal{T}}}$ is the subset of elements that are refined from $\mathcal{T}$ to $\tilde{\mathcal{T}}$ and $C_4$ is independent of the mesh size.
\end{Lemma}
\begin{proof}
Note that $V_\mathcal{T}\subset V_{\tilde{\mathcal{T}}}$. From (\ref{OCP_opt_h}) we have
\begin{eqnarray}
(\alpha u_\mathcal{T}+S^*_\mathcal{T}(y_\mathcal{T}-y_d),v_\mathcal{T}-u_\mathcal{T})\geq 0\quad \forall v_\mathcal{T}\in U_{ad}, \nonumber\\
(\alpha u_{\tilde{\mathcal{T}}}+S^*_{\tilde{\mathcal{T}}}(y_{\tilde{\mathcal{T}}}-y_d),v_{\tilde{\mathcal{T}}}-u_{\tilde{\mathcal{T}}})\geq 0\quad \forall v_{\tilde{\mathcal{T}}}\in U_{ad}. \nonumber
\end{eqnarray}
Setting $v_\mathcal{T}= u_{\tilde{\mathcal{T}}}$ and $v_{\tilde{\mathcal{T}}} = u_\mathcal{T}$ in above inequalities and adding them together, we conclude from (\ref{OCP_state_h}) and (\ref{OCP_adjoint_h})  that
\begin{eqnarray}
&&\alpha\|u_\mathcal{T}- u_{\tilde{\mathcal{T}}}\|^2_{0,\Omega}\leq (S^*_\mathcal{T}(y_\mathcal{T}-y_d)-S^*_{\tilde{\mathcal{T}}}(y_{\tilde{\mathcal{T}}}-y_d),u_{\tilde{\mathcal{T}}}- u_\mathcal{T})\nonumber\\
&=& (S^*_\mathcal{T}(y_\mathcal{T}-y_d)-S^*_{\tilde{\mathcal{T}}}(y_\mathcal{T}-y_d),u_{\tilde{\mathcal{T}}}- u_\mathcal{T})+(S^*_{\tilde{\mathcal{T}}}(y_\mathcal{T}-y_d)-S^*_{\tilde{\mathcal{T}}}(y_{\tilde{\mathcal{T}}}-y_d),u_{\tilde{\mathcal{T}}}-u_\mathcal{T})\nonumber\\
&=&(S^*_\mathcal{T}(y_\mathcal{T}-y_d)-S^*_{\tilde{\mathcal{T}}}(y_\mathcal{T}-y_d),u_{\tilde{\mathcal{T}}}- u_\mathcal{T})+(S_\mathcal{T}u_\mathcal{T}-S_{\tilde{\mathcal{T}}}u_{\tilde{\mathcal{T}}},S_{\tilde{\mathcal{T}}}u_{\tilde{\mathcal{T}}}-S_{\tilde{\mathcal{T}}}u_\mathcal{T})\nonumber\\
&=&(S^*_\mathcal{T}(y_\mathcal{T}-y_d)-S^*_{\tilde{\mathcal{T}}}(y_\mathcal{T}-y_d),u_{\tilde{\mathcal{T}}}- u_\mathcal{T})+(S_\mathcal{T}u_\mathcal{T}-S_{\tilde{\mathcal{T}}}u_{\tilde{\mathcal{T}}},S_{\tilde{\mathcal{T}}}u_{\tilde{\mathcal{T}}}-S_\mathcal{T}u_\mathcal{T})\nonumber\\
&&+(S_\mathcal{T}u_\mathcal{T}-S_{\tilde{\mathcal{T}}}u_{\tilde{\mathcal{T}}},S_\mathcal{T} u_\mathcal{T}-S_{\tilde{\mathcal{T}}}u_\mathcal{T}).\label{u_rep_k}
\end{eqnarray}
It follows from Young's inequality that
\begin{eqnarray}
&&\alpha\|u_\mathcal{T}-u_{\tilde{\mathcal{T}}}\|^2_{0,\Omega}+\|y_\mathcal{T}-y_{\tilde{\mathcal{T}}}\|^2_{0,\Omega}\nonumber\\
&\leq&  C\|S_{\tilde{\mathcal{T}}}u_\mathcal{T}-S_\mathcal{T} u_\mathcal{T}\|^2_{0,\Omega}+C\|S^*_{\tilde{\mathcal{T}}}(y_\mathcal{T}-y_d)-S^*_\mathcal{T}(y_\mathcal{T}-y_d)\|_{0,\Omega}^2,\label{u_error_k}
\end{eqnarray}
where we used the fact that $y_{\tilde{\mathcal{T}}}=S_{\tilde{\mathcal{T}}}u_{\tilde{\mathcal{T}}}$ and $y_\mathcal{T}=S_\mathcal{T} u_\mathcal{T}$. Moreover, from the triangle inequality, the discrete stability of elliptic equation and (\ref{u_error_k}) we infer that
\begin{eqnarray}
\|p_\mathcal{T}-p_{\tilde{\mathcal{T}}}\|_{0,\Omega}&\leq& \|S^*_{\tilde{\mathcal{T}}}(y_\mathcal{T}-y_d)-S^*_{\tilde{\mathcal{T}}}(y_{\tilde{\mathcal{T}}}-y_d)\|_{0,\Omega}+\|S^*_{\tilde{\mathcal{T}}}(y_\mathcal{T}-y_d)-S^*_\mathcal{T}(y_\mathcal{T}-y_d)\|_{0,\Omega}\nonumber\\
&\leq&C\|y_\mathcal{T}-y_{\tilde{\mathcal{T}}}\|_{0,\Omega}+\|S^*_{\tilde{\mathcal{T}}}(y_\mathcal{T}-y_d)-S^*_\mathcal{T}(y_\mathcal{T}-y_d)\|_{0,\Omega}\nonumber\\
&\leq&  C\|S_{\tilde{\mathcal{T}}}u_\mathcal{T}-S_\mathcal{T} u_\mathcal{T}\|_{0,\Omega}+C\|u_\mathcal{T}-u_{\tilde{\mathcal{T}}}\|_{0,\Omega}+C\|S^*_{\tilde{\mathcal{T}}}(y_\mathcal{T}-y_d)-S^*_\mathcal{T}(y_\mathcal{T}-y_d)\|_{0,\Omega}\nonumber\\
&\leq&  C\|S_{\tilde{\mathcal{T}}}u_\mathcal{T}-S_\mathcal{T} u_\mathcal{T}\|_{0,\Omega}+C\|S^*_{\tilde{\mathcal{T}}}(y_\mathcal{T}-y_d)-S^*_\mathcal{T}(y_\mathcal{T}-y_d)\|_{0,\Omega}.\label{p_error_k}
\end{eqnarray}
Note that $S_\mathcal{T} u_\mathcal{T}$ and $S_{\tilde{\mathcal{T}}}u_\mathcal{T}$ are the finite element approximations of $Su_\mathcal{T}$ on $V_\mathcal{T}$ and $V_{\tilde{\mathcal{T}}}$ associated with partitions $\mathcal{T}$ and $\tilde{\mathcal{T}}$, respectively. Similarly,  $S^*_\mathcal{T}(y_\mathcal{T}-y_d)$ and $S^*_{\tilde{\mathcal{T}}}(y_\mathcal{T}-y_d)$ are the finite element approximations of $S^*(y_\mathcal{T}-y_d)$ on $V_\mathcal{T}$ and $V_{\tilde{\mathcal{T}}}$. Then from Lemma 2 in \cite{Demlow} we conclude that
\begin{eqnarray}
\|S_\mathcal{T}u_\mathcal{T}-S_{\tilde{\mathcal{T}}}u_\mathcal{T}\|_{0,\Omega}\leq C\eta_{\mathcal{T},y} (u_\mathcal{T}, S_\mathcal{T}u_\mathcal{T},\mathcal{R}_{\mathcal{T}\rightarrow \tilde{\mathcal{T}}}),\label{localized_y}\\
\|S_\mathcal{T}^*(y_\mathcal{T}-y_d)-S_{\tilde{\mathcal{T}}}^*(y_\mathcal{T}-y_d)\|_{0,\Omega}\leq C\eta_{\mathcal{T},p} (y_\mathcal{T}-y_d, S_\mathcal{T}^*(y_\mathcal{T}-y_d),\mathcal{R}_{\mathcal{T}\rightarrow \tilde{\mathcal{T}}}).\label{localized_p}
\end{eqnarray}
Combining (\ref{u_error_k}), (\ref{p_error_k}), (\ref{localized_y}) and (\ref{localized_p}) we finish the proof.
\end{proof}

To bound the number of marked elements in D\"orfler's marking we give the following lemma. \begin{Lemma}\label{Lemma:13}
Let $\theta<\frac{1}{C_2(1+C_4)}$. Let $\mu$ be sufficiently small such that Lemma \ref{Lemma:15} is valid. If 
\begin{eqnarray}
&&\|u-u_\mathcal{T}\|_{0,\Omega}+\|y-y_\mathcal{T}\|_{0,\Omega}+\|p-p_\mathcal{T}\|_{0,\Omega}+{\rm osc}_{\mathcal{T}}\nonumber\\
&\leq&(1-\theta C_2(1+C_4))(\|u-u_k\|_{0,\Omega}+\|y-y_k\|_{0,\Omega}+\|p-p_k\|_{0,\Omega}+{\rm osc}_k)
\end{eqnarray}
holds for $\mathcal{T}_k\in \mathbb{T}_\mu$ and $\mathcal{T}_k\subset\mathcal{T}\in \mathbb{T}$, then we have
\begin{eqnarray}
\eta_k(\mathcal{R}_{\mathcal{T}_k\rightarrow\mathcal{T}})\geq \theta\eta_k(\mathcal{T}_k).
\end{eqnarray}
\end{Lemma}
\begin{proof}
From the triangle inequality we have
\begin{eqnarray}
\|u-u_k\|_{0,\Omega}+\|y-y_k\|_{0,\Omega}+\|p-p_k\|_{0,\Omega}&\leq& \|u_\mathcal{T}-u_k\|_{0,\Omega}+\|y_\mathcal{T}-y_k\|_{0,\Omega}+\|p_\mathcal{T}-p_k\|_{0,\Omega}\nonumber\\
&&+\|u-u_\mathcal{T}\|_{0,\Omega}+\|y-y_\mathcal{T}\|_{0,\Omega}+\|p-p_\mathcal{T}\|_{0,\Omega},\nonumber\\
{\rm osc}_k&\leq&{\rm osc}_k(\mathcal{R}_{\mathcal{T}_k\rightarrow\mathcal{T}})+{\rm osc}_{\mathcal{T}}.\nonumber
\end{eqnarray}
It follows from Theorem \ref{Thm:1}, Lemma \ref{Lemma:15} and the assumption of the lemma that
\begin{eqnarray}
\theta(C_4+1)\eta_k(\mathcal{T}_k)&\leq& \theta C_2(1+C_4)(\|u-u_k\|_{0,\Omega}+\|y-y_k\|_{0,\Omega}+\|p-p_k\|_{0,\Omega}+{\rm osc}_k)\nonumber\\
&\leq&\|u-u_k\|_{0,\Omega}+\|y-y_k\|_{0,\Omega}+\|p-p_k\|_{0,\Omega}+{\rm osc}_k\nonumber\\
&&-\|u-u_\mathcal{T}\|_{0,\Omega}-\|y-y_\mathcal{T}\|_{0,\Omega}-\|p-p_\mathcal{T}\|_{0,\Omega}-{\rm osc}_{\mathcal{T}}\nonumber\\
&\leq&\|u_\mathcal{T}-u_k\|_{0,\Omega}+\|y_\mathcal{T}-y_k\|_{0,\Omega}+\|p_\mathcal{T}-p_k\|_{0,\Omega}+{\rm osc}_k(\mathcal{R}_{\mathcal{T}_k\rightarrow\mathcal{T}})\nonumber\\
&\leq&(1+C_4)\eta_k(\mathcal{R}_{\mathcal{T}_k\rightarrow\mathcal{T}}),\nonumber
\end{eqnarray}
this completes the proof.
\end{proof}

\begin{Lemma}\label{Lemma:14}
 For any $s>0$, let $(y,p)\in\mathcal{A}^s$. Assume that $\mu$ is sufficiently small such that Theorem \ref{Thm:2} is valid. Then under the assumptions of Lemma \ref{Lemma:13}, the number of marked elements $\mathcal{M}_k\subset\mathcal{T}_k$ defined in Algorithm \ref{Alg:3.0} satisfies
\begin{eqnarray}
\#\mathcal{M}_k\lesssim |(y,p)|_{\mathcal{A}^s}^{1\over s}(\|u-u_k\|_{0,\Omega}+\|y-y_k\|_{0,\Omega}+\|p-p_k\|_{0,\Omega}+{\rm osc}_{k})^{-{1\over s}}.\label{complexity}
\end{eqnarray}
\end{Lemma}
\begin{proof}
Let 
\begin{eqnarray}
\varepsilon =\frac{1-\theta C_2(1+C_4)}{C_5}(\|u-u_k\|_{0,\Omega}+\|y-y_k\|_{0,\Omega}+\|p-p_k\|_{0,\Omega}+{\rm osc}_{k}),
\end{eqnarray}
where $C_5$ is defined in Remark \ref{Rem:1}. Let $\mathcal{T}'\in \mathbb{T}$ and $  y_{\mathcal{T}'},  p_{\mathcal{T}'}\in V_{\mathcal{T}'}$ such that
\begin{eqnarray}
&&\|y-  y_{\mathcal{T}'}\|_{0,\Omega}+{\rm osc}(y_{\mathcal{T}'}-y_d,\mathcal{T}')+{\rm osc}(\mathcal{L}y_{\mathcal{T}'},\mathcal{T}')\nonumber\\
&&+\|p-  p_{\mathcal{T}'}\|_{0,\Omega}+{\rm osc}(f+P_{[a,b]}(p_{\mathcal{T}'}),\mathcal{T}')+{\rm osc}(\mathcal{L}^*p_{\mathcal{T}'},\mathcal{T}')\leq \varepsilon.
\end{eqnarray}
We can conclude from the definition of $\mathcal{A}^s$ that 
\begin{eqnarray}
\#\mathcal{T}'-\#\mathcal{T}_0\lesssim |(y,p)|_{\mathcal{A}^s}^{1\over s}\varepsilon^{-{1\over s}}.\label{comp_1}
\end{eqnarray}
It can be shown from Appendix A in \cite{Demlow} that $\mathcal{T}'$ can be refined to a partition $\mathcal{T}''\in \mathbb{T}_\mu$ with $\#\mathcal{T}''-\#\mathcal{T}_0\lesssim \#\mathcal{T}'-\#\mathcal{T}_0$ depending on $\mu$. Let $\mathcal{T}:=\mathcal{T}''\oplus \mathcal{T}_k$ be the smallest common refinement of $\mathcal{T}''$ and $\mathcal{T}_k$ in $\mathbb{T}_\mu$, there holds $\#\mathcal{T}-\#\mathcal{T}_k\leq \#\mathcal{T}''-\#\mathcal{T}_0$ (see \cite{Stevenson}). In view of $V_{\mathcal{T}'}\subset V_\mathcal{T}$, the monotonicity of data oscillation (\cite[Remark 2.1]{Cascon}) and Remark \ref{Rem:1} we have
\begin{eqnarray}
&&\|u-u_\mathcal{T}\|_{0,\Omega}+\|y-y_\mathcal{T}\|_{0,\Omega}+\|p-p_\mathcal{T}\|_{0,\Omega}+{\rm osc}_{\mathcal{T}}\nonumber\\
&\leq& C_5\Big(\inf\limits_{v_{\mathcal{T}'}\in V_{\mathcal{T}'}}(\|y-v_{\mathcal{T}'}\|_{0,\Omega}+{\rm osc}(v_{\mathcal{T}'}-y_d,\mathcal{T}')+{\rm osc}(\mathcal{L}v_{\mathcal{T}'},\mathcal{T}'))\nonumber\\
&&+\inf\limits_{w_{\mathcal{T}'}\in V_{\mathcal{T}'}}(\|p-w_{\mathcal{T}'}\|_{0,\Omega}+{\rm osc}(f+P_{[a,b]}(w_{\mathcal{T}'}),\mathcal{T}')+{\rm osc}(\mathcal{L}^*w_{\mathcal{T}'},\mathcal{T}'))\Big)\nonumber\\
&\leq&(1-\theta C_2(1+C_4))(\|u-u_k\|_{0,\Omega}+\|y-y_k\|_{0,\Omega}+\|p-p_k\|_{0,\Omega}+{\rm osc}_{k}),
\end{eqnarray}
so we can conclude from Lemma \ref{Lemma:13} that $\eta_k(\mathcal{R}_{\mathcal{T}_k\rightarrow\mathcal{T}})\geq \theta\eta_k(\mathcal{T}_k)$. Note that Algorithm \ref{Alg:3.0} selects a minimal set $\mathcal{M}_k$ of $\mathcal{T}_k$ with $\eta_k(\mathcal{M}_k)\geq \theta \eta_k(\mathcal{T}_k)$. Thus, 
\begin{eqnarray}
\#\mathcal{M}_k\leq \#\mathcal{R}_{\mathcal{T}_k\rightarrow\mathcal{T}}\leq \#\mathcal{T}-\#\mathcal{T}_k\leq \#\mathcal{T}''-\#\mathcal{T}_0\lesssim \#\mathcal{T}'-\#\mathcal{T}_0.
\end{eqnarray}
Now (\ref{complexity}) follows from (\ref{comp_1}). 
\end{proof}

Now we are in the position to present our final result on the quasi-optimality of AFEM for solving optimal control problems.
\begin{Theorem}
Let $(u,y,p)\in  U_{ad}\times H_0^1(\Omega)\times H_0^1(\Omega)$ be the solution of optimal control problem (\ref{OCP})-(\ref{OCP_state}) and $(u_n,y_n,p_n) \in U_{ad}\times  V_n\times V_n$ be the solution of the discrete problem (\ref{OCP_state_h})-(\ref{OCP_adjoint_h}) generated by the adaptive Algorithm \ref{Alg:3.1}. For some $s>0$, let $(y,p)\in\mathcal{A}^s$. Under the assumptions of Lemma \ref{Lemma:13} there holds that
\begin{eqnarray}
\|u-u_n\|_{0,\Omega}+\|y-y_n\|_{0,\Omega}+\|p-p_n\|_{0,\Omega}+{\rm osc}_n\lesssim  |(y,p)|_{\mathcal{A}^s}(\#\mathcal{T}_n-\#\mathcal{T}_0)^{-s}\label{quasi_complexity}
\end{eqnarray}
provided $h_0\ll 1$.
\end{Theorem}
\begin{proof}
It follows from Lemmas \ref{Lemma:15} and \ref{Lemma:14} that 
\begin{eqnarray}
\#\mathcal{T}_n-\#\mathcal{T}_0&\lesssim& \sum\limits_{k=0}^{n-1}\#\mathcal{M}_k\nonumber\\
&\lesssim&|(y,p)|_{\mathcal{A}^s}^{1\over s}\sum\limits_{k=0}^{n-1}(\|u-u_k\|_{0,\Omega}+\|y-y_k\|_{0,\Omega}+\|p-p_k\|_{0,\Omega}+{\rm osc}_{k})^{-{1\over s}}.
\end{eqnarray}
Due to Theorem \ref{Thm:3} we obtain for $0\leq k<n$ that
\begin{eqnarray}
&&\|u-u_{n}\|_{0,\Omega}+\|y-y_{n}\|_{0,\Omega}+\|p-p_{n}\|_{0,\Omega}+{\rm osc}_{n}\nonumber\\
&\lesssim& \nu^{n-k}\Big(\|u-u_{k}\|_{0,\Omega}+\|y-y_{k}\|_{0,\Omega}+\|p-p_{k}\|_{0,\Omega}+{\rm osc}_{k}\Big).\nonumber
\end{eqnarray}
Thus, 
\begin{eqnarray}
\#\mathcal{T}_n-\#\mathcal{T}_0&\lesssim&|(y,p)|_{\mathcal{A}^s}^{1\over s}(\|u-u_n\|_{0,\Omega}+\|y-y_n\|_{0,\Omega}+\|p-p_n\|_{0,\Omega}+{\rm osc}_{n})^{-{1\over s}}\sum\limits_{k=0}^{n-1}\nu^{\frac{n-k}{s}}\nonumber\\
&\lesssim &|(y,p)|_{\mathcal{A}^s}^{1\over s}(\|u-u_n\|_{0,\Omega}+\|y-y_n\|_{0,\Omega}+\|p-p_n\|_{0,\Omega}+{\rm osc}_{n})^{-{1\over s}},
\end{eqnarray}
where the last inequality holds due to $\nu<1$. This completes the proof of the theorem.
\end{proof}

\section{Numerical experiments}
\setcounter{equation}{0}
In this section we carry out some numerical tests in two dimensions to support our theoretical results obtained in this paper. As indicated in \cite{Demlow}, the additional refinement of elements to ensure the mesh grading property (\ref{mesh_grading}) seems to be not necessary in practical computations to deliver optimal convergence of $L^2$-norm based AFEM. So in current paper we use the practical bisection algorithm without additional refinement, similar phenomenon can be observed for the optimal convergence. Moreover, in the following examples we set $\mathcal{L}=-\Delta$.

\begin{Example}\label{Exm:1}
We consider an example defined on $\Omega=(0,1)^2$. We set $\alpha = 0.1$  and choose the exact solutions as 
\begin{eqnarray}
y&=&\arctan(-50x_1+100x_2-25),\nonumber\\
p&=&16x_1(1-x_1)x_2(1-x_2)(1+\arctan(200({1\over 16}-(x_1-0.5)^2-(x_2-0.5)^2))),\nonumber\\
u&=&\max\{-5,\min\{-1,-{p\over \alpha}\}\},\nonumber
\end{eqnarray}
the corresponding $f$ and $y_d$ can be obtained after simple calculation. Note that we impose inhomogeneous Dirichlet boundary condition for the state $y$ and our theoretical analysis applies to this case after some simple adaptations .
\end{Example}

In Figure \ref{fig:1} we plot the profiles of the state and adjoint state variables on adaptively refined mesh with $\theta=0.3$ and $20$ adaptive loops. Although the solutions are smooth, larger gradients can be observed in certain areas so the adaptive finite element method may deliver much smaller errors compared to the uniform refinement. In Figure \ref{fig:2} and \ref{fig:3} we show the adaptive meshes after $15$, $20$ and $25$ adaptive loops with D\"orfler's marking parameter $\theta=0.3$. We can see that the meshes concentrate on the areas where the solutions have larger gradients. Moreover, in Figure \ref{fig:3} we plot the active sets of both the continuous control, the discrete controls by using variational control discretization and piecewise linear and continuous finite element approximations. We can observe that the active set of variational discretized control is more close to the continuous one  compared with full control discretization, this shows the superiority of variational control discretization. In Figure \ref{fig:4} we give the comparisons of convergence history of the $L^2$-norm errors of the optimal control, the state and adjoint state and the error estimators on uniformly refined meshes ($\theta$=1) and adaptively refined meshes with $\theta =0.3$ and $\theta=0.4$, respectively. Although optimal convergence of second order in $L^2$-norm can be observed for both the uniform refinement and adaptive refinement, we have smaller errors for adaptive algorithm which shows the power of AFEM for problems even with $H^2$-regularity.

\begin{figure}[ht]
\centering
\includegraphics[width=7cm,height=7cm]{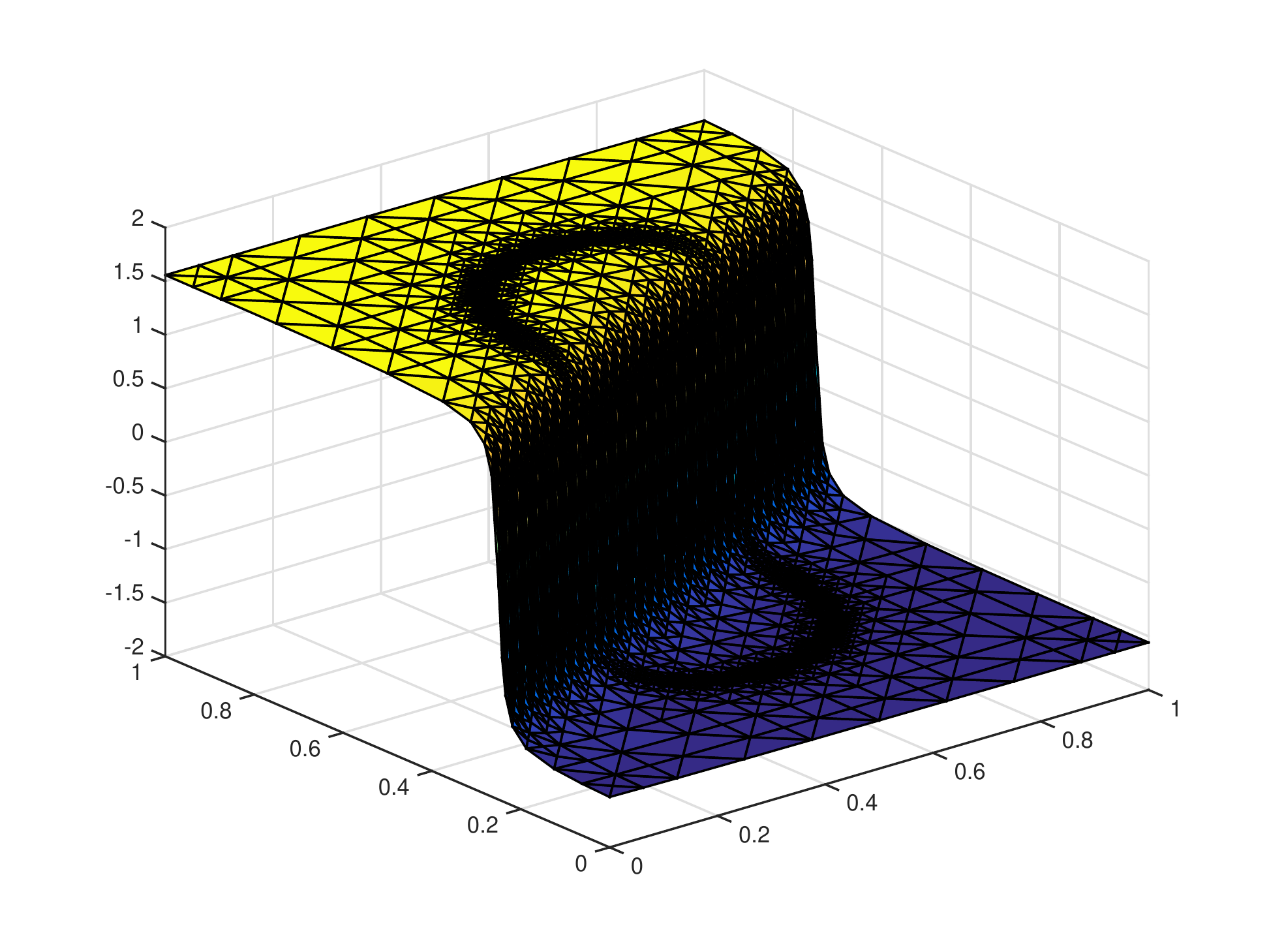}
\includegraphics[width=7cm,height=7cm]{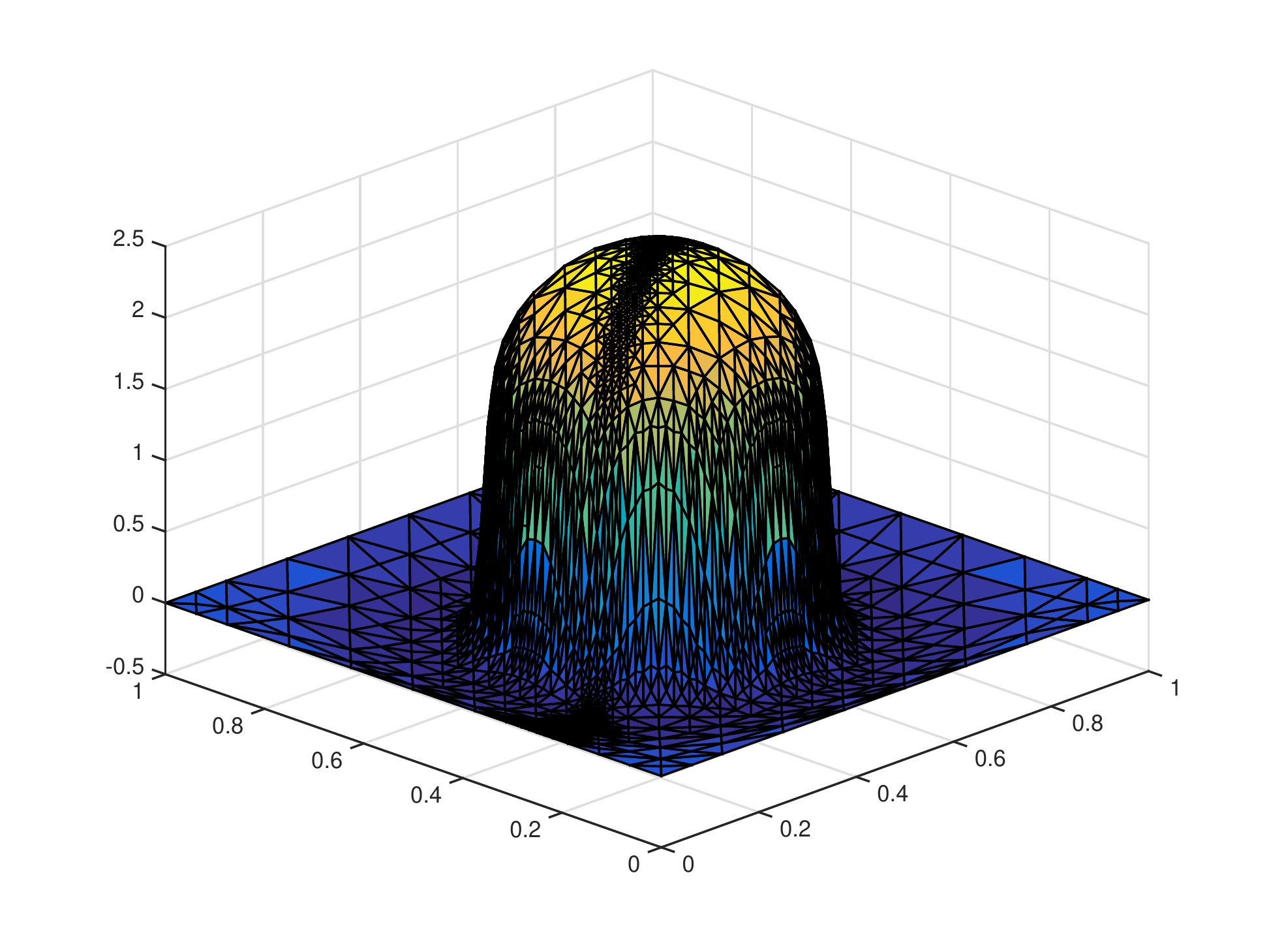}
\caption{The profiles of the numerical state (left) and adjoint state (right) on adaptively refined mesh with $\theta=0.3$ and $20$ adaptive loops for Example \ref{Exm:1} generated by Algorithm \ref{Alg:3.1}.}
\label{fig:1}
\end{figure}

\begin{figure}[ht]
\centering
\includegraphics[width=7cm,height=7cm]{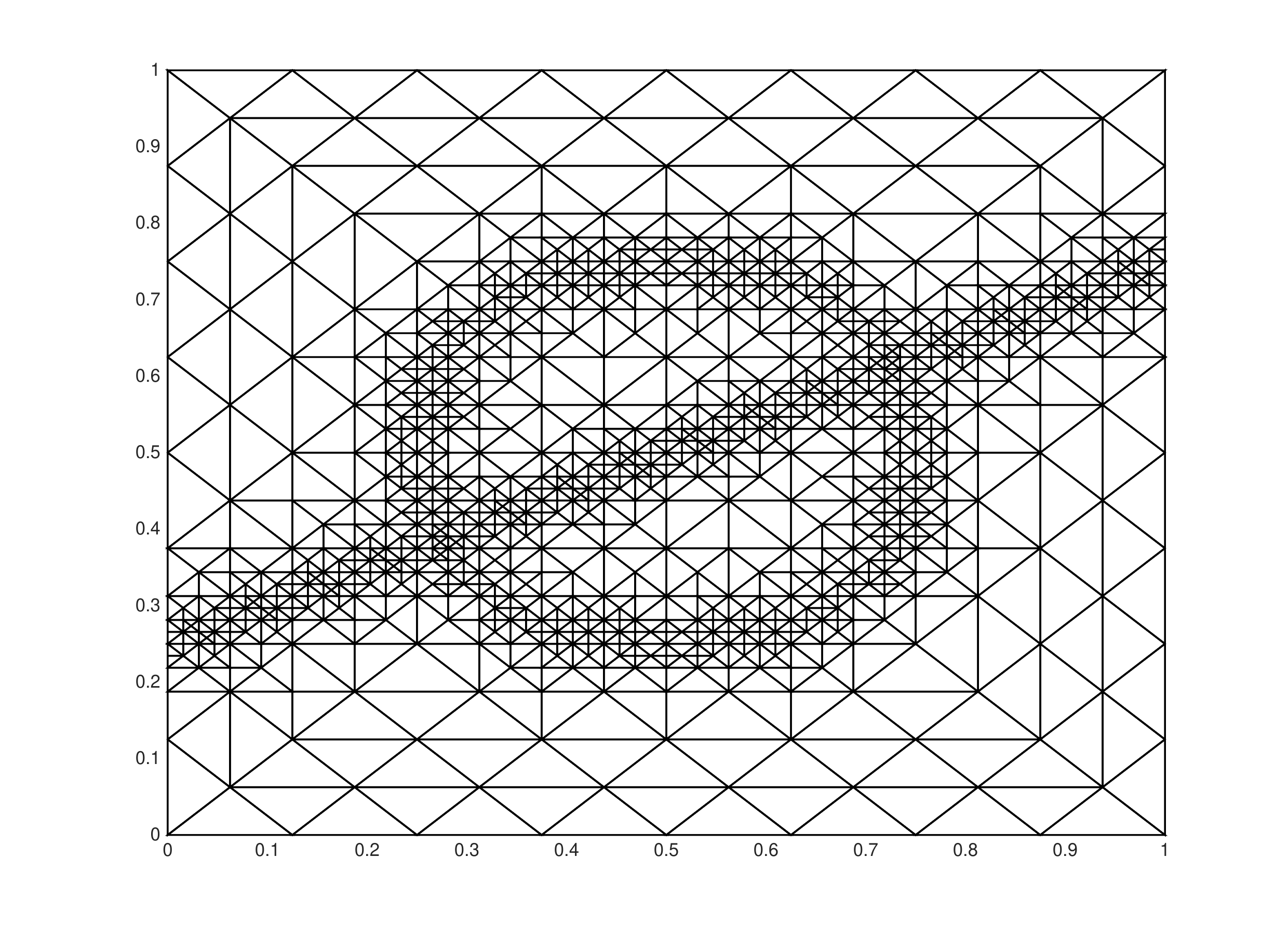}
\includegraphics[width=7cm,height=7cm]{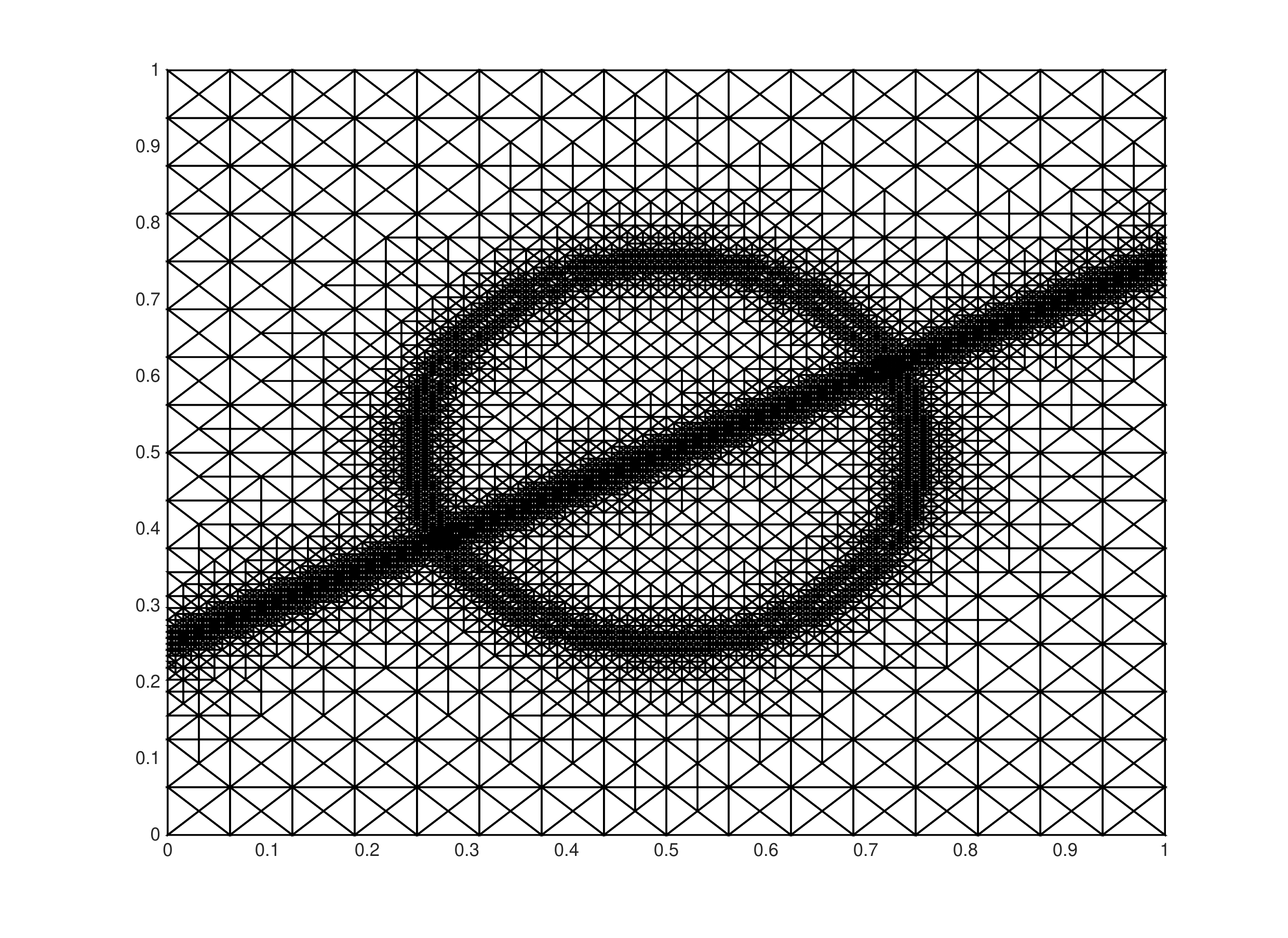}
\caption{The adaptive meshes after $15$ steps (left) and $25$ steps for Example \ref{Exm:1} generated by Algorithm \ref{Alg:3.1} with $\theta =0.3$.}
\label{fig:2}
\end{figure}

\begin{figure}[ht]
\centering
\includegraphics[width=7cm,height=7cm]{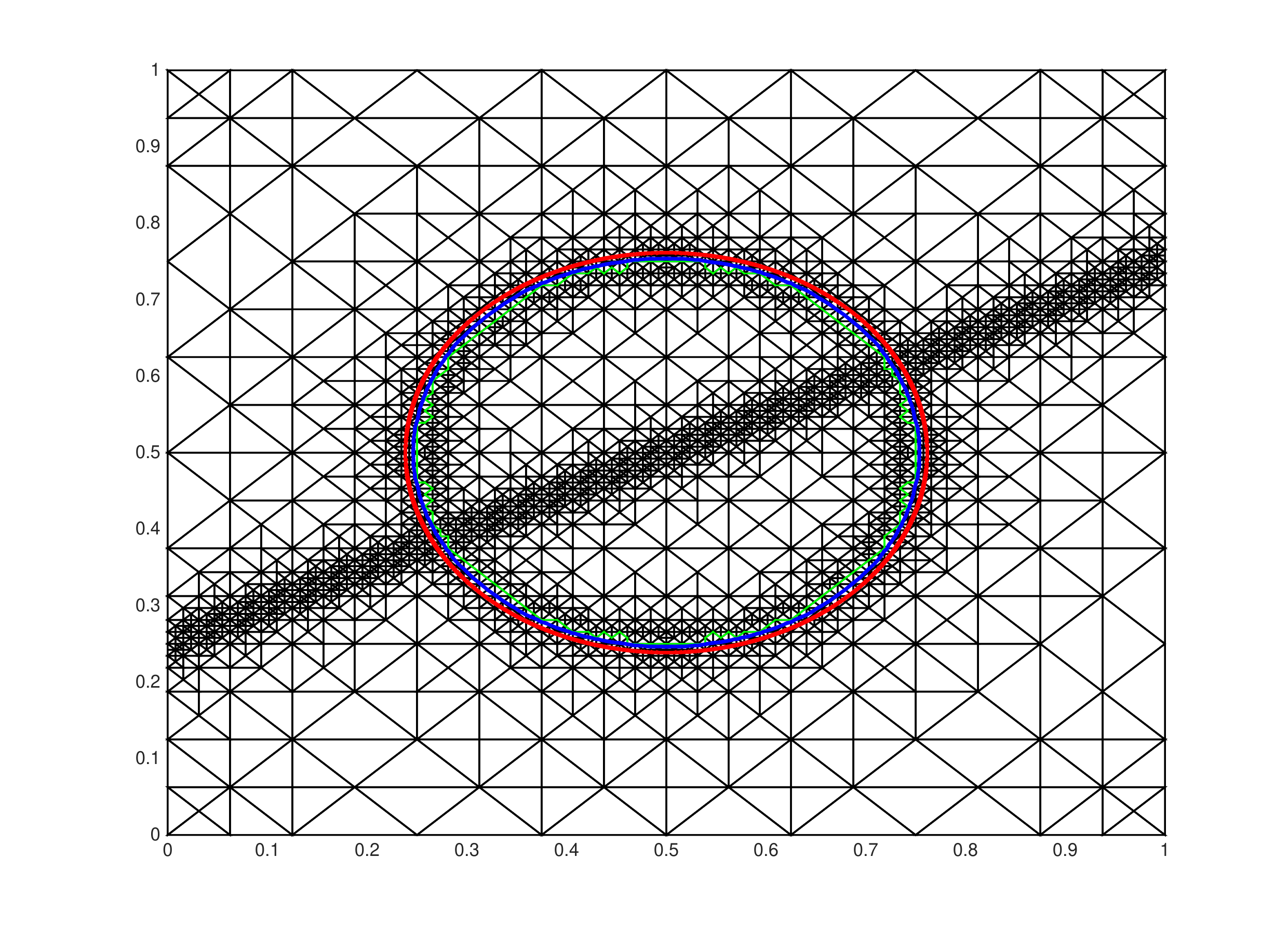}
\includegraphics[width=7cm,height=7cm]{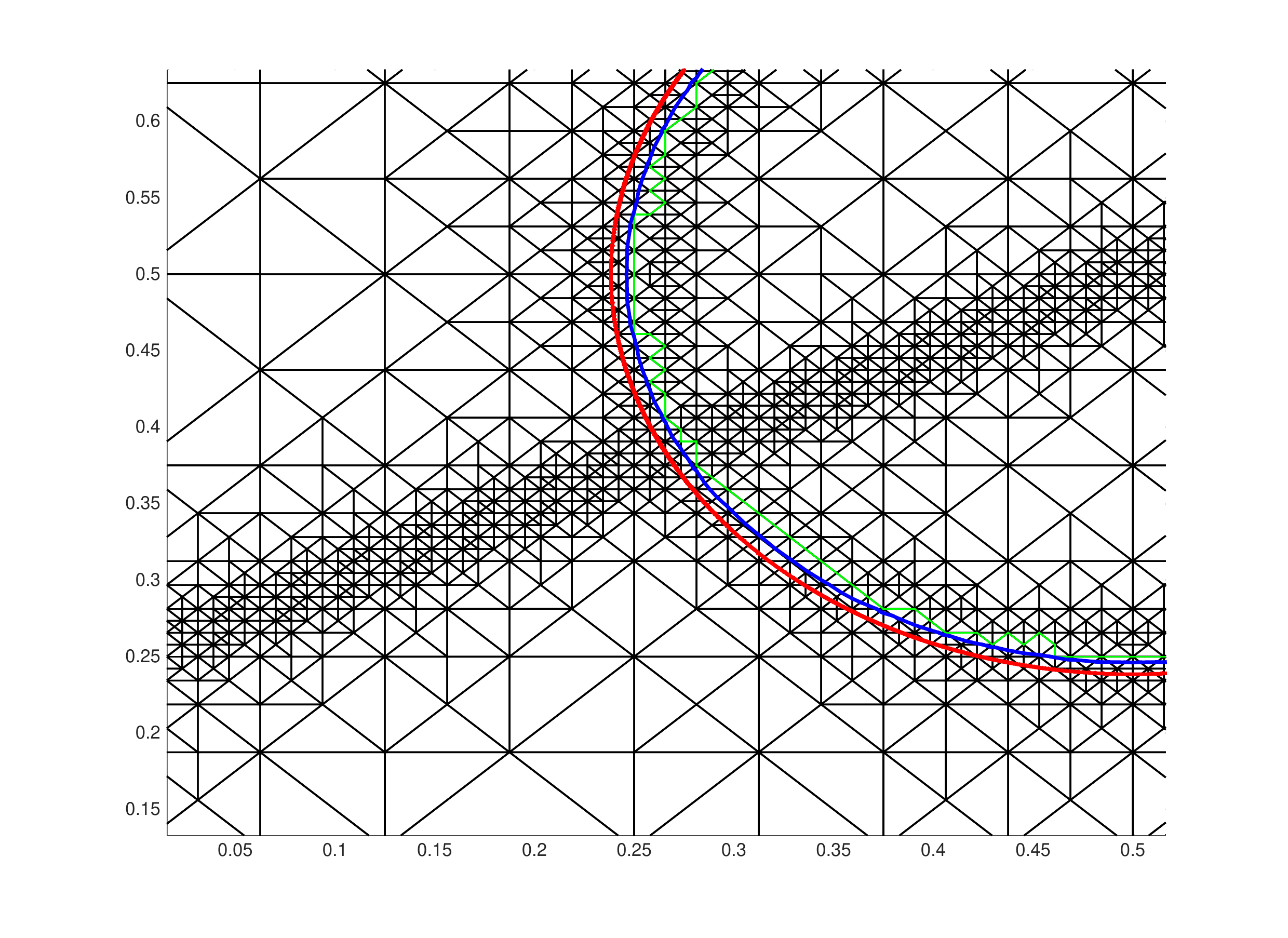}
\caption{The adaptive meshes after $20$ steps (left) and zoom in near the left below corner for Example \ref{Exm:1} generated by Algorithm \ref{Alg:3.1} with $\theta =0.3$. The red line depicts the boarder of the active set of the continuous solution, the blue line depicts the boarder of the active set when using variational discretization, and the green line depicts the boarder of the active set obtained by using piecewise linear, continuous controls. }
\label{fig:3}
\end{figure}

\begin{figure}[ht]
\centering
\includegraphics[width=7cm,height=7cm]{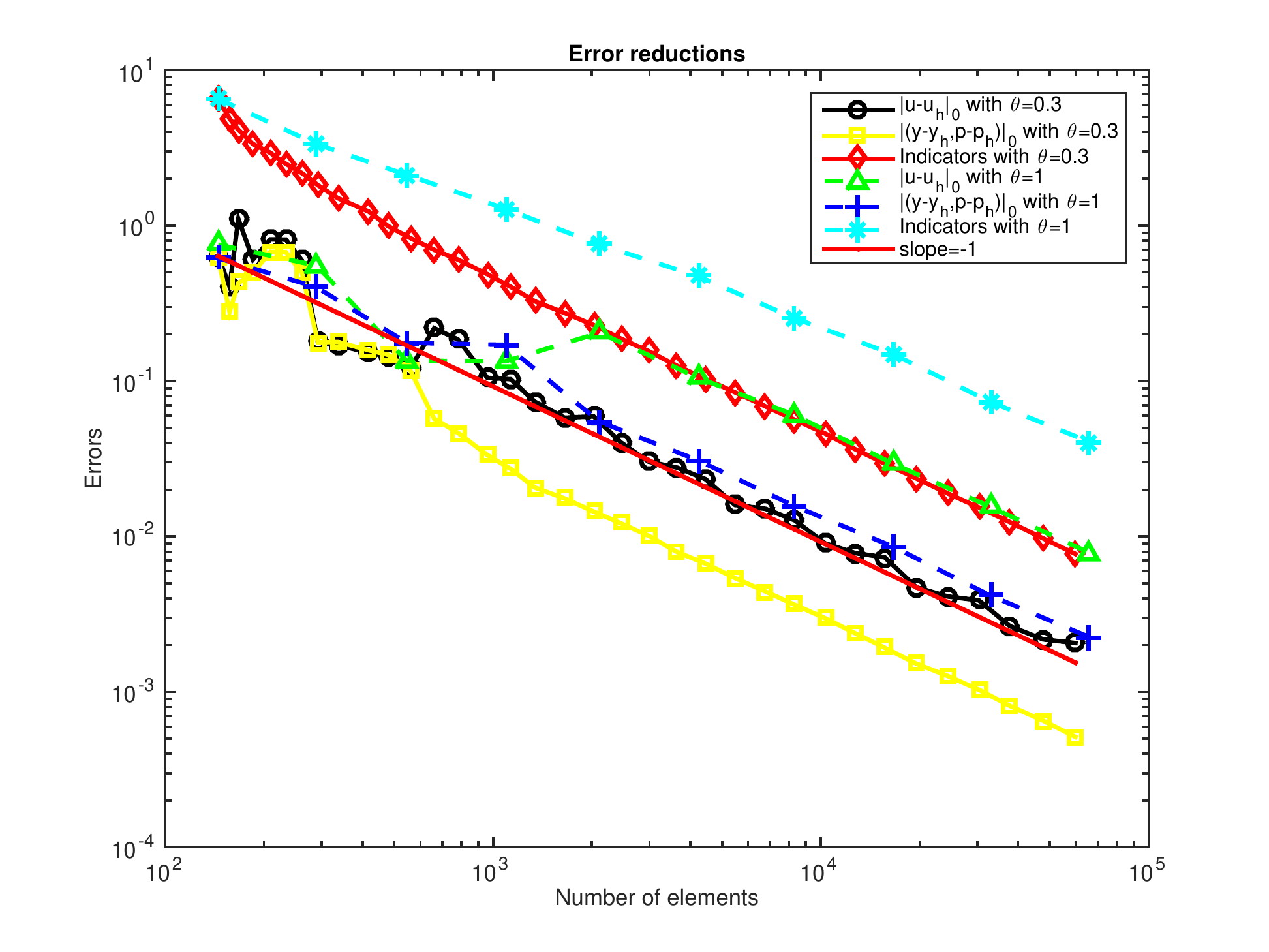}
\includegraphics[width=7cm,height=7cm]{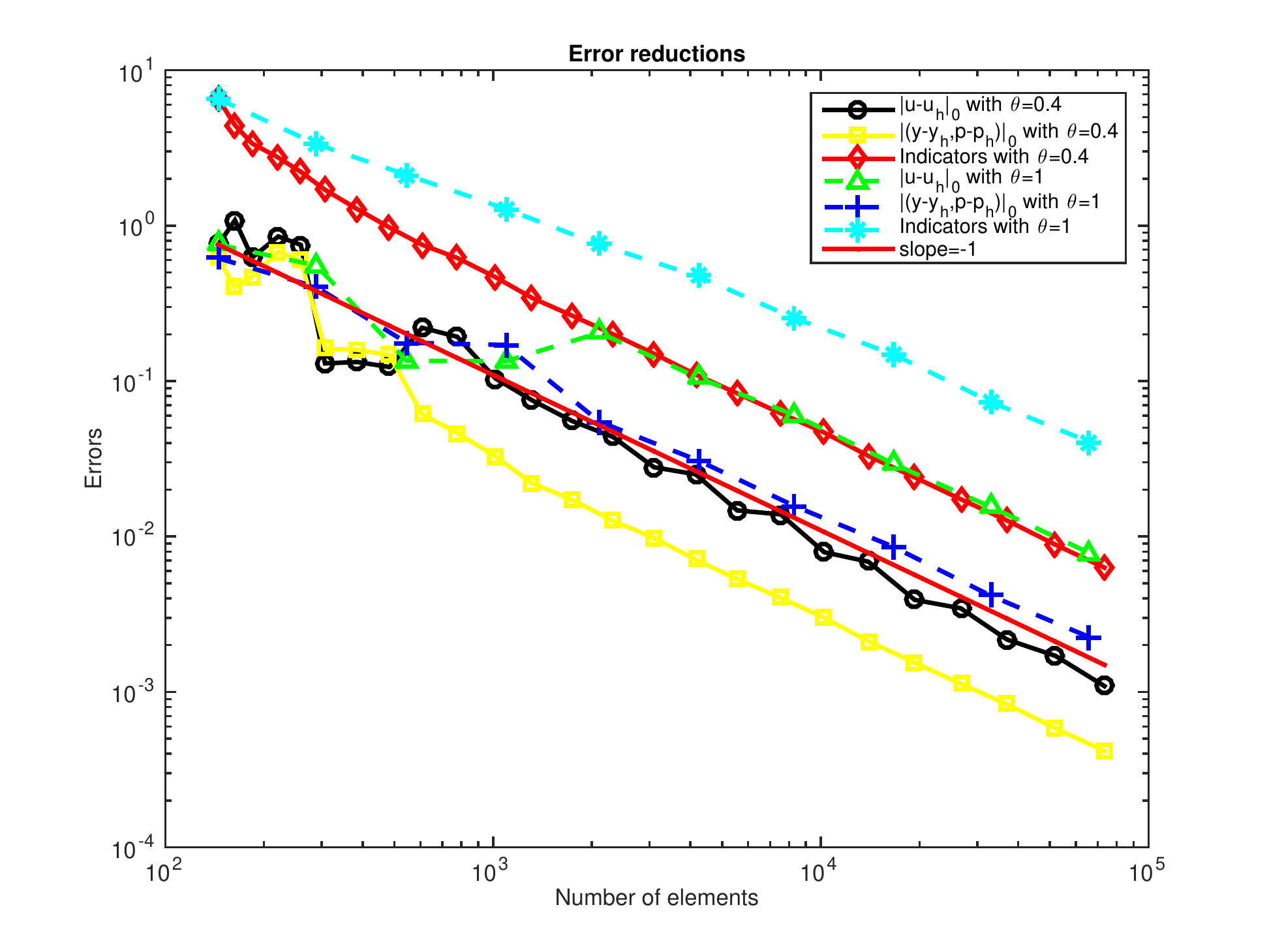}
\caption{The comparisons of convergence history of the optimal control, state and adjoint state and the error estimators on uniformly refined meshes ($\theta$=1) and adaptively refined meshes with $\theta =0.3$ (left) and $\theta=0.4$ (right), respecvitely, for Example \ref{Exm:1} generated by Algorithm \ref{Alg:3.1}.}
\label{fig:4}
\end{figure}



\begin{Example}\label{Exm:2}
In the second example we consider an optimal control problem without explicit solutions. We set $\Omega=(0,1)^2$, $\alpha = 10^{-2}$, $a=10$ and $b=15$. We choose the singular $f$  and desired state $y_d$  as
\begin{eqnarray}
f=\sqrt{(x_1-1)^2+(x_2-1)^2}^{-1.5},\quad y_d=\sqrt{x_1^2+x_2^2}^{-1.9}.\nonumber
\end{eqnarray}
Note that $f$ and $y_d$ are not in $L^2(\Omega)$ as we assumed in the paper, the theory we derived above does not apply in this case. However, since the singularity is only located in two points, only some simple modifications need be made in the computations. We intend to use this example to explain that even in convex domain the solution of elliptic equation may have singularity caused by singular data so the adaptive FEM may also find applications.  
\end{Example}

In Figure \ref{fig:5} we plot the profiles of the state and adjoint state variables on adaptively refined mesh with $\theta=0.3$ and $20$ adaptive loops. Since $f$ and $y_d$ have singularities near the points $(0,0)$ and $(1,1)$, respectively, we can observe the corresponding singularities for the state and adjoint state. In Figure \ref{fig:6} and \ref{fig:7} we show the adaptive meshes after $15$, $20$ and $25$ adaptive loops with D\"orfler's marking parameter $\theta=0.3$. We can see that the meshes concentrate on the points $(0,0)$ and $(1,1)$ where the singularities of the solutions are located. Moreover, in Figure \ref{fig:7} we plot the active sets of the discrete controls by using variational control discretization and piecewise linear and continuous finite element approximations. We can observe that the active set of variational discretized control crosses the elements and can give better results, as indicated in Example \ref{Exm:1}. Since we do not have explicit solutions, in Figure \ref{fig:8} we only show the comparisons of convergence history of the error estimators on uniformly refined meshes ($\theta$=1) and adaptively refined meshes with $\theta =0.3$ and $\theta=0.4$, respectively. We can observe the optimal second order convergence for the reduction of the error estimators for the adaptive refinement which reflects the optimal convergence of the $L^2$-norm of the control, the state and adjoint state because the a posteriori error estimate is reliable and efficient. Note that only reduced orders are observed for the uniform refinement, which is due to the singularity of the solutions caused by singular data.

\begin{figure}[ht]
\centering
\includegraphics[width=7cm,height=7cm]{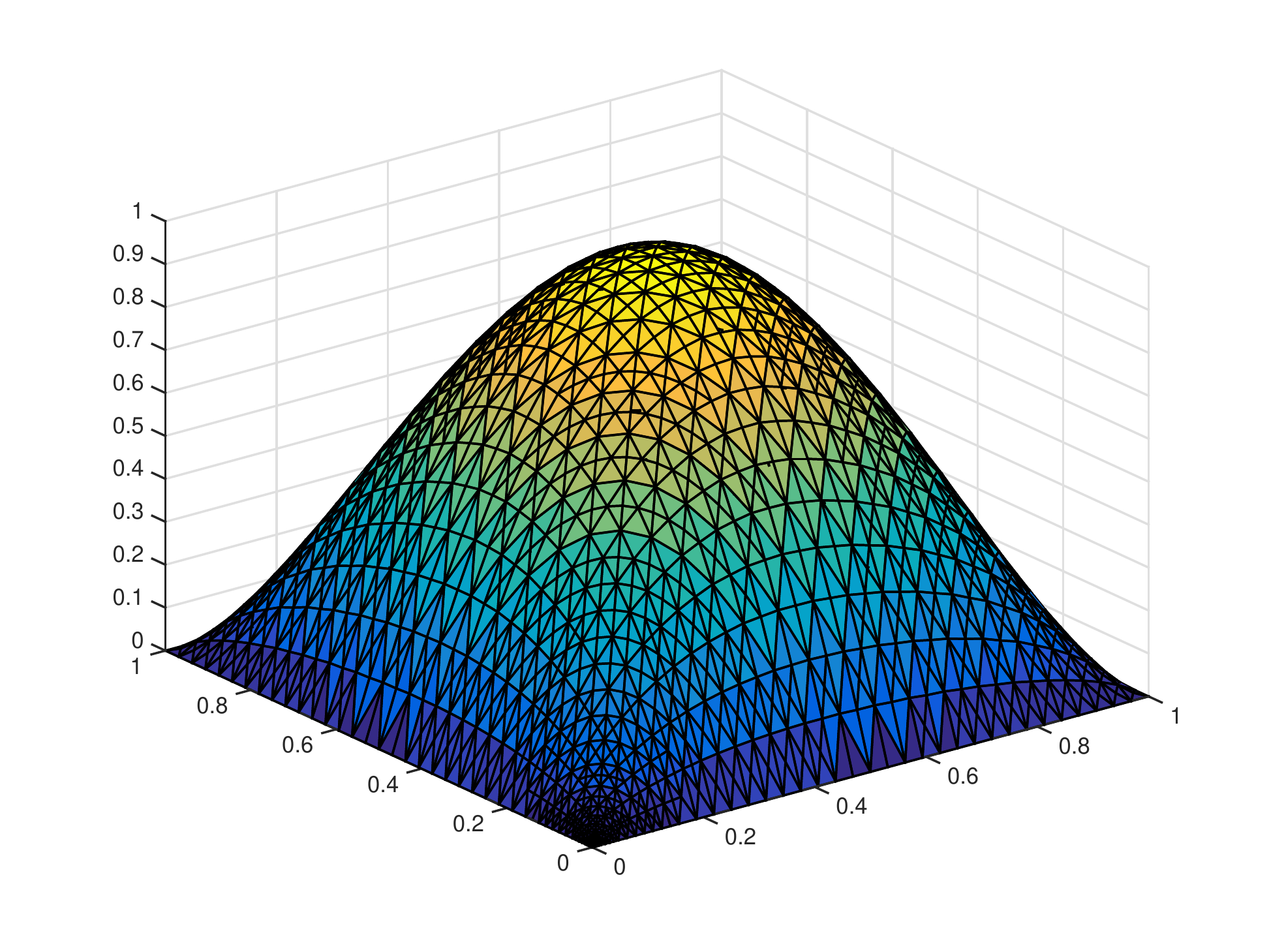}
\includegraphics[width=7cm,height=7cm]{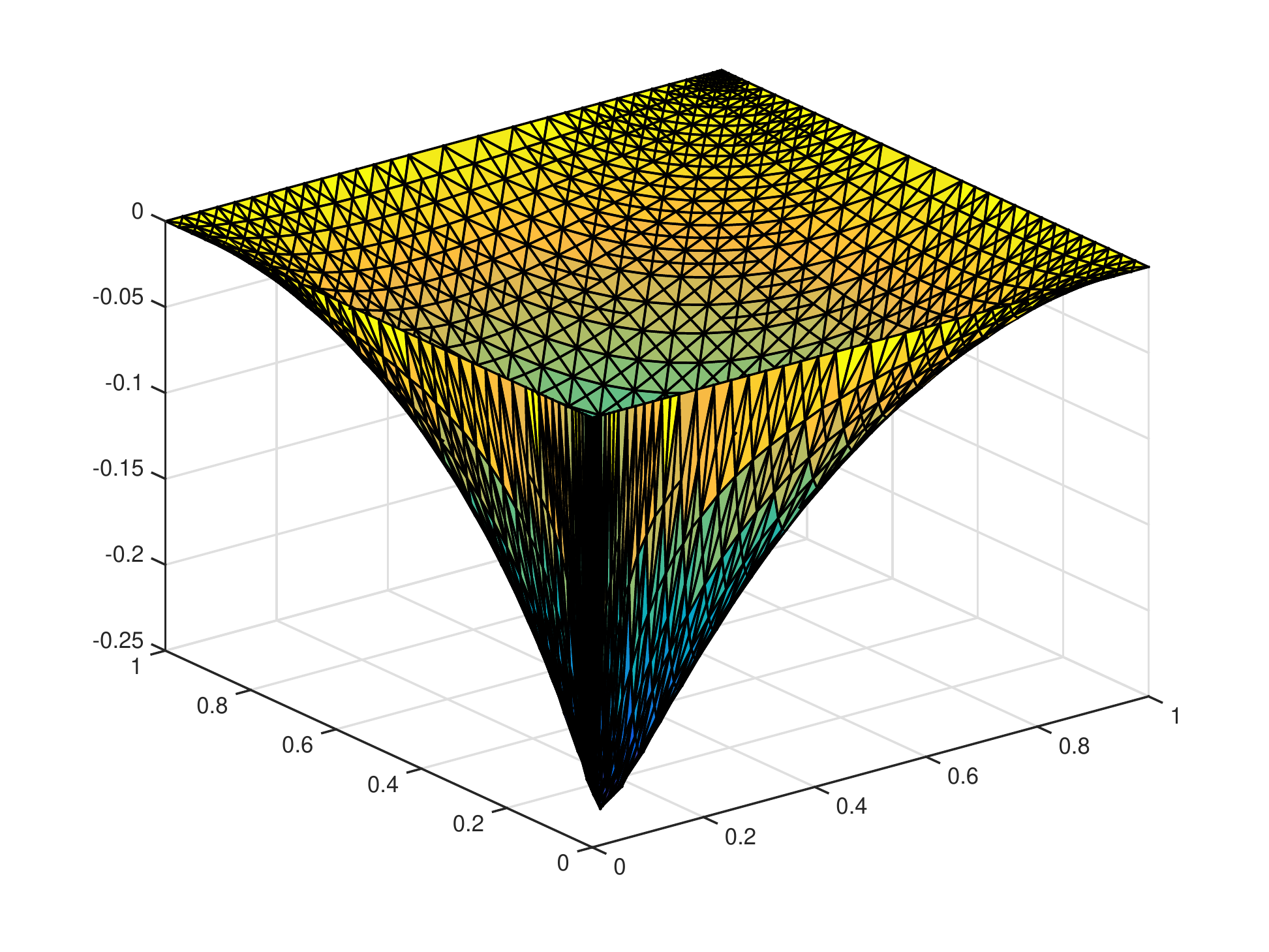}
\caption{The profiles of the numerical state (left) and adjoint state (right) on adaptively refined mesh with $\theta=0.3$ and $20$ adaptive loops for Example \ref{Exm:2} generated by Algorithm \ref{Alg:3.1}.}
\label{fig:5}
\end{figure}

\begin{figure}[ht]
\centering
\includegraphics[width=7cm,height=7cm]{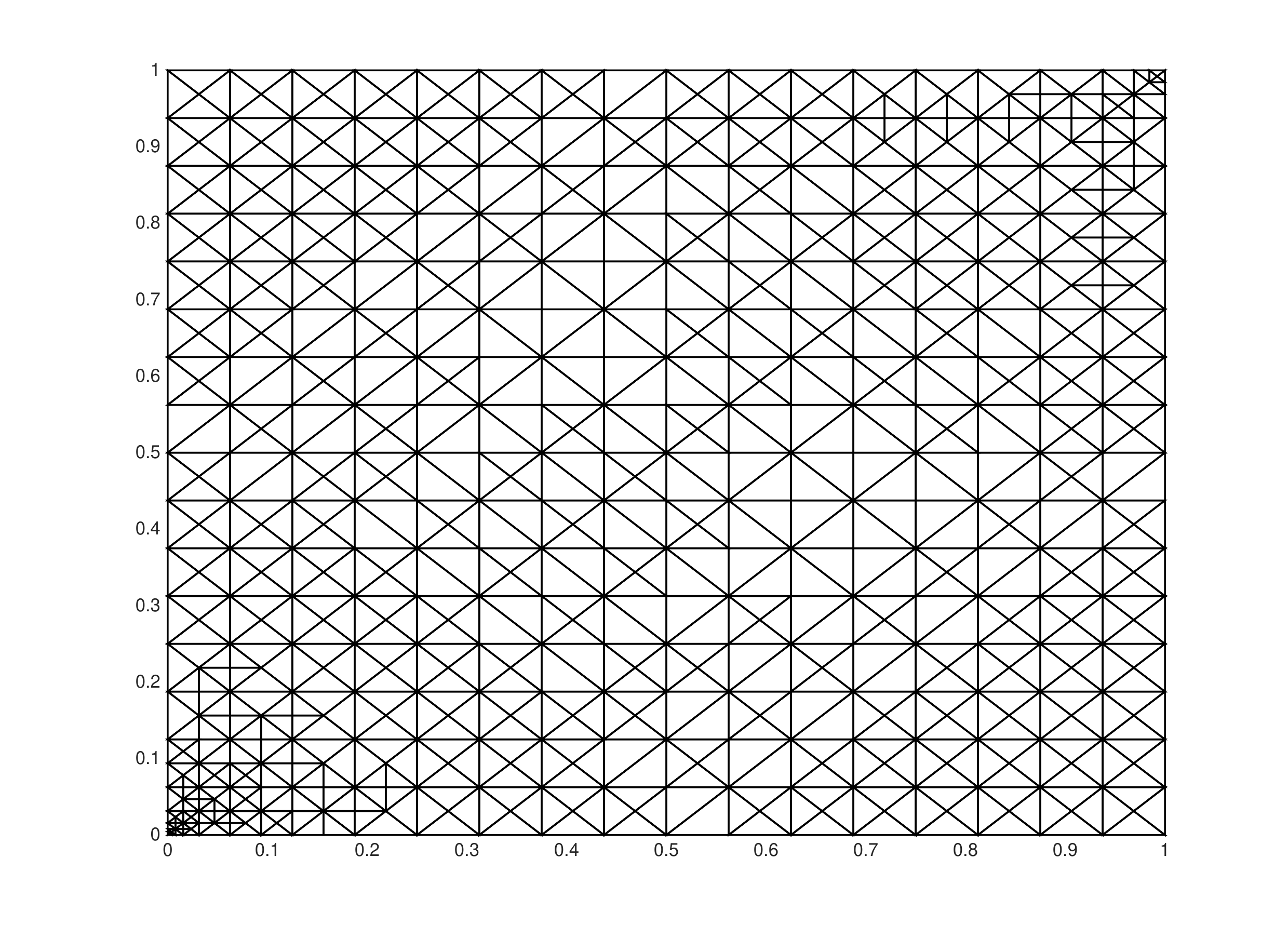}
\includegraphics[width=7cm,height=7cm]{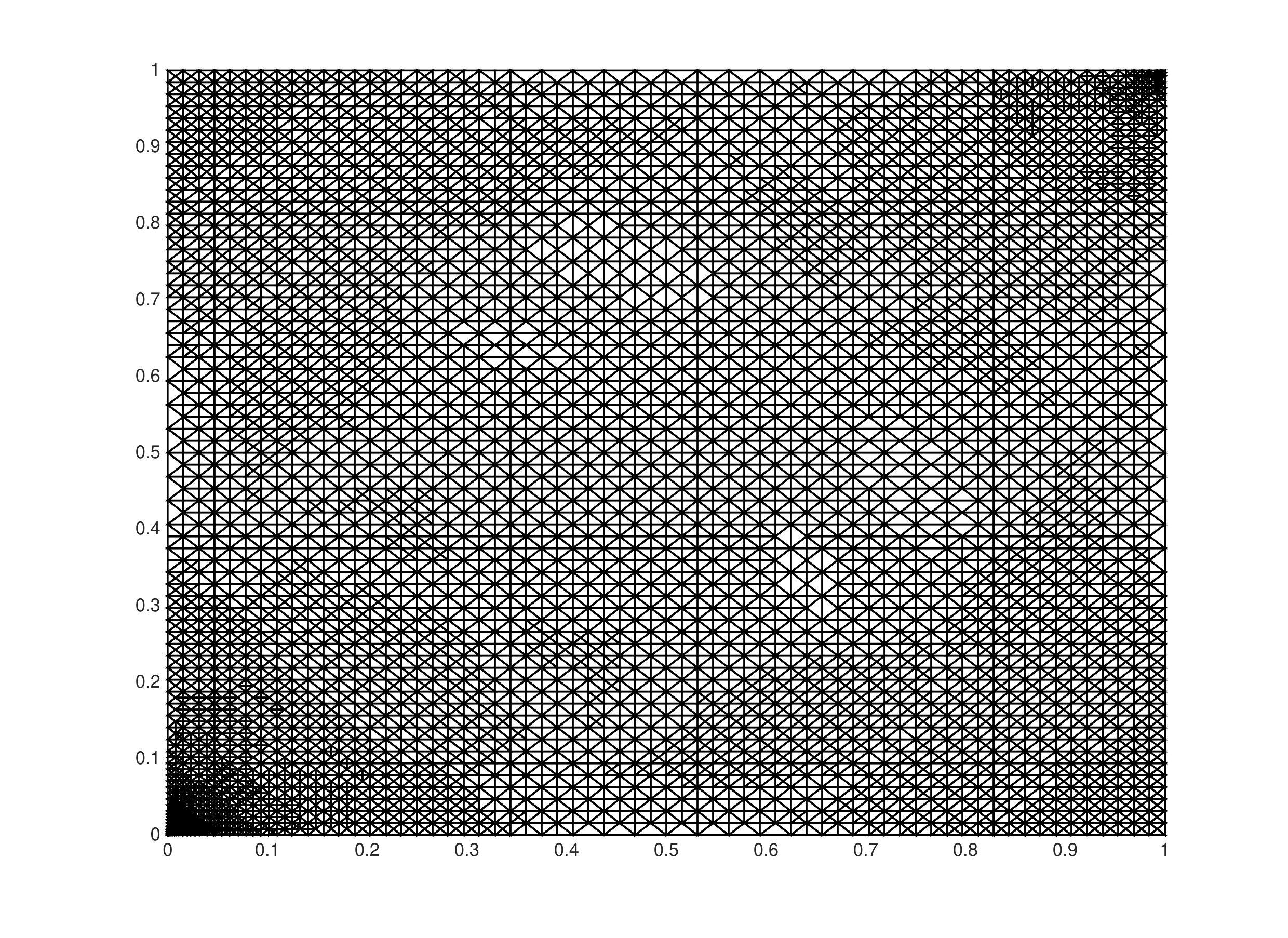}
\caption{The adaptive meshes after $15$ steps (left) and $25$ steps for Example \ref{Exm:2} generated by Algorithm \ref{Alg:3.1} with $\theta =0.3$.}
\label{fig:6}
\end{figure}

\begin{figure}[ht]
\centering
\includegraphics[width=7cm,height=7cm]{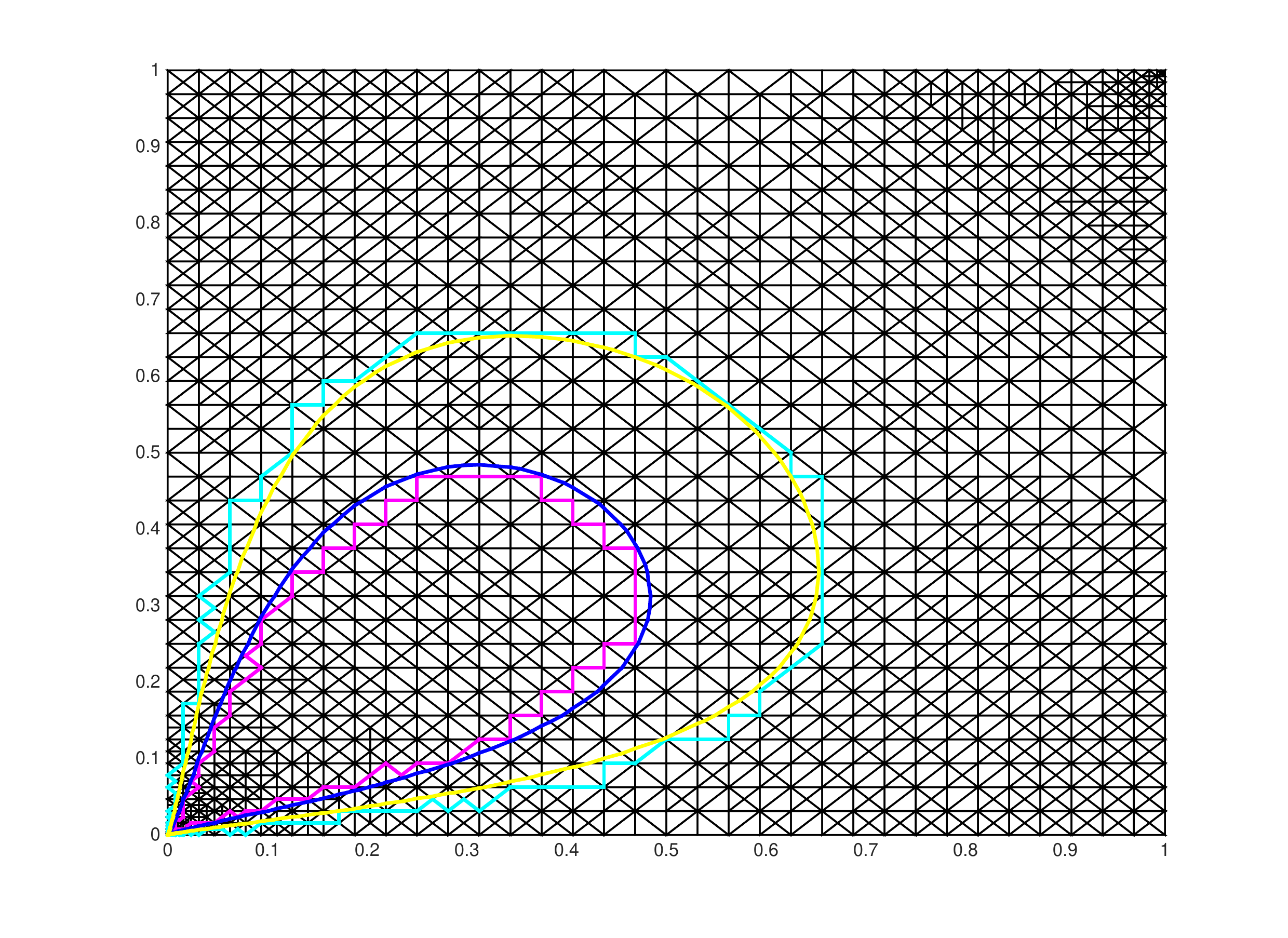}
\includegraphics[width=7cm,height=7cm]{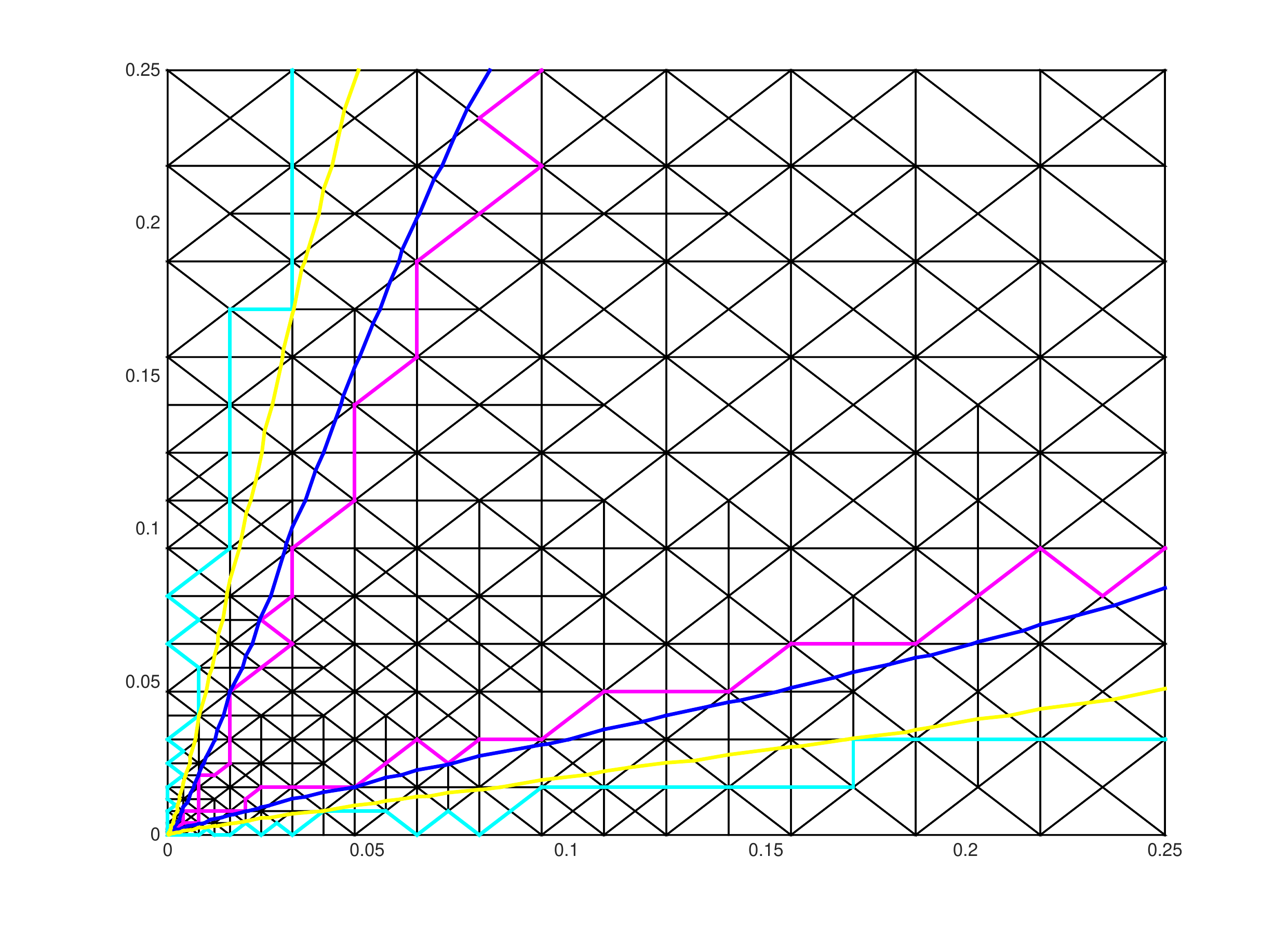}
\caption{The adaptive meshes after $20$ steps (left) and zoom in near the origin for Example \ref{Exm:2} generated by Algorithm \ref{Alg:3.1} with $\theta =0.3$. The blue (upper bound) and yellow (lower bound) lines depict the boarder of the active set of the discrete solution when using variational discretization, and the pink (upper bound) and cyan (lower bound) lines depict the boarder of the active set obtained by using piecewise linear, continuous controls. }
\label{fig:7}
\end{figure}

\begin{figure}[ht]
\centering
\includegraphics[width=7cm,height=7cm]{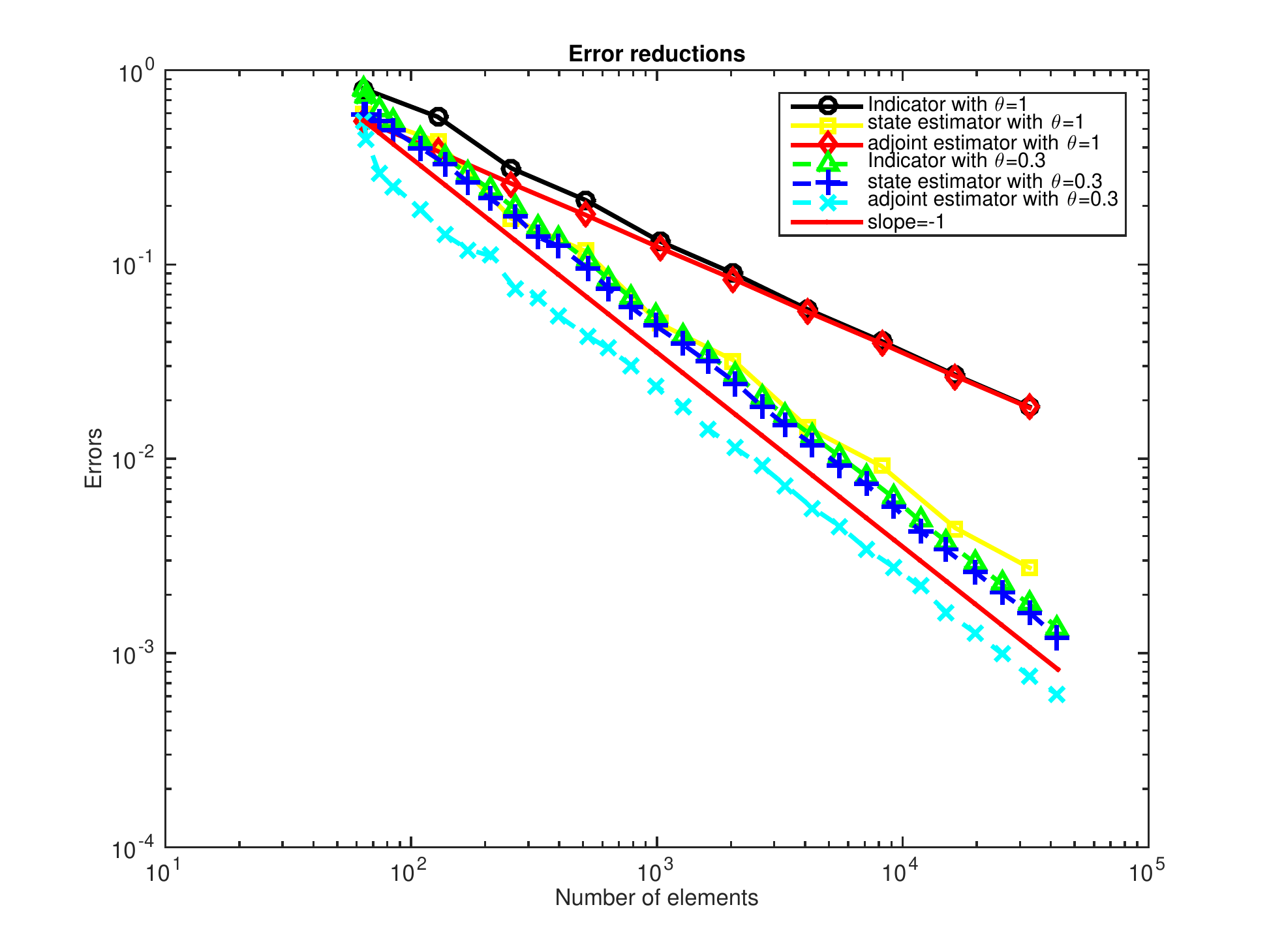}
\includegraphics[width=7cm,height=7cm]{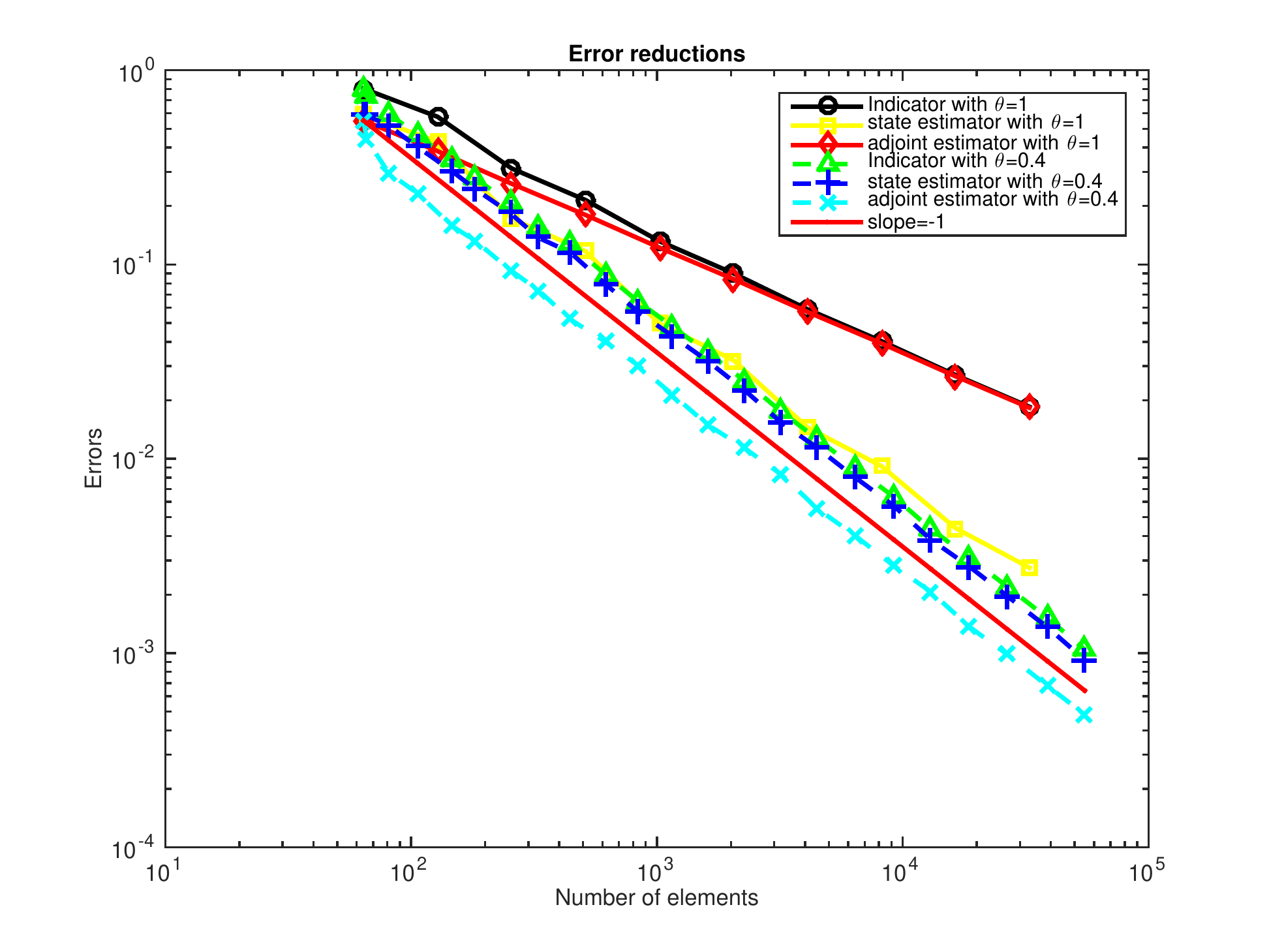}
\caption{The comparisons of convergence history of the optimal control, state and adjoint state and the error estimators on uniformly refined meshes ($\theta$=1) and adaptively refined meshes with $\theta =0.3$ (left) and $\theta=0.4$ (right), respecvitely, for Example \ref{Exm:2} generated by Algorithm \ref{Alg:3.1}.}
\label{fig:8}
\end{figure}


\section*{Acknowledgements}
The first author was supported by the National Basic Research Program of China under grant 2012CB821204 and the National Natural Science Foundation of China under grants 11671391 and 91530204. The second author acknowledged the support of the National Natural Science Foundation of China under grants 11571356 and 91530204. The third author acknowledged the support of the National Natural Science Foundation of China under grant 11301311.


 \medskip

\end{document}